%% file: jiang23a.tex
\documentclass[final,12pt]{colt2023} %

\title[A Quasi-Newton Method with Global Non-Asymptotic Superlinear Convergence]{Online Learning Guided Curvature Approximation: A Quasi-Newton Method with Global Non-Asymptotic Superlinear Convergence}
\usepackage{times}
\usepackage{algorithm,algorithmic}
\usepackage{eqparbox}
\definecolor{comment}{RGB}{2,128, 9}

\floatstyle{ruled}
\newfloat{subroutine}{htbp}{lok}
\floatname{subroutine}{Subroutine}

\usepackage[shortlabels]{enumitem}
\usepackage{nicefrac}
\usepackage{cancel}
\input{math_commands.tex}

\usepackage{tikz}
\usetikzlibrary{decorations.pathreplacing,calc}
\newcommand{\tikzmark}[1]{\tikz[overlay,remember picture] \node (#1) {};}

\newcommand*{\AddNote}[4]{%
    \begin{tikzpicture}[overlay, remember picture]
        \draw [decoration={brace,amplitude=0.3em},decorate,thick,black]
            ($(#3.east)!(#1.north)!($(#3.east)-(0,1)$)$) --  
            ($(#3.east)!(#2.south)!($(#3.east)-(0,1)$)$)
                node [align=left, text width=4cm, pos=0.5, anchor=west] {#4};
    \end{tikzpicture}
}

\newcommand{\calC}{\mathcal{C}}

\newtheorem{assumption}{Assumption}
\coltauthor{\Name{Ruichen Jiang} \Email{rjiang@utexas.edu}\\
 \Name{Qiujiang Jin} \Email{qiujiang@austin.utexas.edu}\\
 \Name{Aryan Mokhtari} \Email{mokhtari@austin.utexas.edu}\\
 \addr The University of Texas at Austin}

\begin{document}

\maketitle

\begin{abstract}%
Quasi-Newton algorithms are among the most popular iterative methods for solving unconstrained minimization problems, largely due to their favorable superlinear convergence property.  
However, existing results for these algorithms are limited as  they provide either (i) a global convergence guarantee with an \textit{asymptotic} superlinear convergence rate, or (ii) a \textit{local} non-asymptotic superlinear rate for the case that the initial point and the initial Hessian approximation are chosen properly. {In particular, no current analysis for quasi-Newton methods guarantees global convergence with an explicit superlinear convergence rate.} 
In this paper, we close this gap and present  
the first \textit{globally} convergent quasi-Newton method with an \emph{explicit non-asymptotic} superlinear convergence rate. 
Unlike classical quasi-Newton methods, we 
build our algorithm upon the hybrid proximal extragradient method and propose a novel \emph{online learning} framework for updating the Hessian approximation matrices. 
Specifically, guided by the convergence analysis, we formulate the Hessian approximation update as an online convex optimization problem in the space of matrices, and we relate the bounded regret of the online problem to the superlinear convergence of our method. 

\end{abstract}


\section{Introduction}
In this paper, we study 
quasi-Newton methods to solve unconstrained optimization problems. This class of algorithms can be viewed as a modification of Newton's method, where  the objective function Hessian is approximated using the gradient information. Specifically, a general template of quasi-Newton methods to minimize a continuously differentiable function $f$ is
\begin{equation}\label{eq:quasi_newton}
  \vx_{k+1} = \vx_k - \rho_k \mB_k^{-1} \nabla f(\vx_k), \qquad k \geq 0,
\end{equation}
where $\rho_k$ is the step size and $\mB_k$ is a matrix that aims to approximate $\nabla^2{f(\vx_k)}$. Several rules for updating $\mB_k$ have been proposed in the literature, and the most prominent include the Davidon-Fletcher-Powell (DFP) method \citep{davidon1959variable,fletcher1963rapidly}, the Broyden-Fletcher-Goldfarb-Shanno (BFGS) method \citep{broyden1970convergence,fletcher1970new,goldfarb1970family,shanno1970conditioning}, and the symmetric rank-one (SR1) method \citep{conn1991convergence,khalfan1993theoretical}. %

The main advantage of quasi-Newton methods is their ability to achieve Q-superlinear convergence under suitable conditions on $f$, i.e., $\lim_{k\rightarrow \infty }\frac{\|\vx_{k+1}-\vx^*\|}{\|\vx_{k}-\vx^*\|} = 0$ where $\vx^*$ is the optimal solution of $f$.  
\cite{broyden1973local,dennis1974characterization} established that DFP and BFGS are locally and Q-superlinearly convergent with unit step size (i.e., $\rho_k=1$ in \eqref{eq:quasi_newton}). 
To ensure global convergence, it is necessary to incorporate quasi-Newton updates with a line search or a trust-region method. \cite{powell1971convergence,dixon1972variable} proved that DFP and BFGS converge globally and Q-superlinearly with an exact line search, which can be computationally prohibitive. Subsequently, \cite{powell1976some} showed that BFGS with an inexact line search retains global and superlinear convergence, and \cite{byrd1987global} later extended the result to the restricted Broyden class except for DFP. Along another line of research,
\cite{conn1991convergence,khalfan1993theoretical,byrd1996analysis} studied the SR1 method in a trust region context and also proved its global and superlinear convergence. 
However, the above results are all of \textit{asymptotic} nature and they fail to provide an explicit upper bound on the distance to the optimal solution after a finite number of iterations. 

To address this shortcoming, several recent papers \citep{rodomanov2021greedy,Rodomanov2021,rodomanov2021new,jin2022non,jin2022sharpened,lin2021greedy,ye2022towards} have studied \textit{local nonasymptotic} superlinear convergence rates of classic quasi-Newton methods or their greedy variants. In particular,  \cite{rodomanov2021new} proved that in a local neighborhood of the optimal solution, if the initial Hessian approximation is set as {$L_1\mB$ where $\mB$ is a positive definite matrix}, BFGS with unit step size converges at a superlinear rate of  
$[e^{\frac{d}{k}\log\frac{L_1}{\mu}}-1]^{k/2}$, where $k$ is the number of iterations, $d$ is the problem dimension, $L_1$ is the smoothness parameter and $\mu$ is the strong convexity parameter {relative to the matrix $\mB$ (see eq.~(26) in \citep{rodomanov2021new}). Note that since $L_1$ and $\mu$ are defined with respect to $\mB$, in general, this superlinear rate  will depend on the condition number of $\mB$.} In a concurrent work, \cite{jin2022non} showed that if the initial Hessian approximation is {{close to the Hessian at the optimal solution or selected as the Hessian at the initial point,}}
BFGS with unit step size can achieve a local superlinear convergence rate of $\left({1}/{k}\right)^{{k}/{2}}$. However, all these results are crucially based on \emph{local analysis}, and there is no clear way of extending these local non-asymptotic superlinear rates into global convergence guarantees for quasi-Newton methods.

{{Specifically, the existing local analyses in both \citep{rodomanov2021new} and \citep{jin2022non} require the initial point $\vx_0$ to be close enough to the optimal solution $\vx^*$, and in this local regime the step size in \eqref{eq:quasi_newton} has to be $\rho_k = 1$. 
Hence, to obtain a global convergence guarantee, it is necessary to use a globalization strategy, such as a line search scheme, and then switch to the local analysis when the iterates reach a local neighborhood of $\vx^*$. However, this approach faces several challenges: (i) It is unclear how to obtain an explicit global convergence rate for quasi-Newton methods with line search. (ii) It is unclear how  to bound the number of iterations before the line search scheme can accept the unit step size $\rho_k = 1$.
(iii) Moreover, regarding the result in \citep{rodomanov2021new},
 it is unclear how to control the condition number of the Hessian approximation matrix when the iterates enter the local neighborhood, which would affect the region of local convergence and the starting moment of superlinear convergence. Similarly, to apply the local analysis in \citep{jin2022non}, the Hessian approximation matrix need be close to the exact Hessian in a local neighborhood, which cannot be guaranteed in general.}}
Hence, the following question remains open: 
\begin{quote}
    \textit{Can we design a globally convergent quasi-Newton method with an explicit superlinear convergence rate?} 
\end{quote}

  In this paper, we answer the above question in affirmative. We propose a novel quasi-Newton proximal extragradient (QNPE) method that achieves  an explicit non-asymptotic superlinear convergence rate. Unlike prior works that use a \textit{local analysis} 
requiring specific conditions on the initial iterate and the initial Hessian approximation,
our \emph{global superlinear} convergence guarantee holds for an arbitrary initialization of the iterate and Hessian approximation.
More precisely, for a $\mu$-strongly convex function $f$ with $L_1$-Lipschitz gradient and $L_2$-Lipschitz Hessian, the iterates $\{\vx_k\}_{k\geq 0}$ generated by our QNPE method satisfy the  following guarantees:

\vspace{1mm}
\noindent\textbf{(i) Global convergence rates.} We have $
    \frac{\|\vx_{k}-\vx^*\|^2}{\|\vx_0-\vx^*\|^2} \!\leq\! \min\bigl\{\big(1+\frac{\mu}{4 L_1}\big)^{-k},\bigl(1+ \frac{\mu}{4L_1} \sqrt{{k}/{C}}\bigr)^{-k}\bigr\}$, where $C = \bigO\Bigl(\frac{\|\mB_0-\nabla^2 f(\vx^*)\|^2_F}{L_1^2}+\frac{L_2^2\|\vx_0-\vx^*\|^2}{\mu L_1}\Bigr) = {\bigO \Bigl( d + \frac{L_2^2\|\vx_0-\vx^*\|^2}{\mu L_1} \Bigr)}$. {Note that the first bound corresponds to a linear convergence rate on par with the rate of gradient descent, while the second one corresponds to a superlinear rate.} 
    In particular, the superlinear rate outperforms the linear rate when $k \geq C$.  %

\vspace{2mm}
\noindent\textbf{(ii) Iteration complexity.} 
  Let $N_\epsilon$ denote the number of iterations required to reach $\epsilon$-accuracy. Then we have $N_{\epsilon} = \bigO\Bigl(\min\Bigl\{\frac{L_1}{\mu}\log\frac{1}{\epsilon},\, {\log\frac{1}{\epsilon}}\Big/{{ \log \Bigl(1+\frac{\mu}{L_1}\left({\frac{L_1}{ C\mu}\log\frac{1}{\epsilon}}\right)^{\nicefrac{1}{3}}\Bigr)}}\Bigr\}\Bigr)$. In particular, in the regime where $\epsilon$ is sufficiently small, we obtain $N_{\epsilon} = \bigO\left(\log\frac{1}{\epsilon}/\log\log\frac{1}{\epsilon}\right)$. 

\vspace{2mm}
\noindent
\textbf{(iii) Computational complexity.}  To achieve $\epsilon$-accuracy, the total number of gradient evaluations and  matrix-vector products is bounded by $3N_{\epsilon}-1$ and $\bigO(N_{\epsilon}\sqrt{\frac{L_1}{\mu}}\log\frac{L_1 N_\epsilon^2d}{\mu\epsilon})$, respectively.

\vspace{1mm}
We obtain these results by taking a quite different route from the existing quasi-Newton literature. Instead of considering an update of the form \eqref{eq:quasi_newton}, we build our method based on the \textit{hybrid proximal extragradient} (HPE) framework \citep{solodov1999hybrid}, which can be interpreted as an inexact variant of the  proximal point method \citep{Martinet1970,Rockafellar1976}. The general HPE method consists of two steps: an inexact proximal point update where $\hat{\vx}_k \approx \vx_k-\eta_k \nabla f(\hat{\vx}_k)$, and an extragradient step where $\vx_{k+1} = \vx_k - \eta_k \nabla f(\hat{\vx}_k)$. In our QNPE method, we implement the first step by using the linear approximation $\nabla f({\vx}_k)+\mB_k(\hat{\vx}_k-\vx_k)$ as a surrogate of $\nabla f(\hat{\vx}_k)$. Considering this approximation and by exploiting strong convexity, the QNPE update is given by
\begin{equation}\label{eq:HPE_update}
  \hat{\vx}_{k} = \vx_k - \eta_k (\mI+\eta_k \mB_k)^{-1} \nabla f(\vx_k), \quad  \vx_{k+1} = \frac{1}{1+2\eta_k\mu}(\vx_k - \eta_k \nabla f(\hat{\vx}_k))+\frac{2\eta_k\mu}{1+2\eta_k\mu} \hat{\vx}_k,
\end{equation}
where $\eta_k$ is the step size and $\mu$ is the strong convexity parameter. 

Moreover, to ensure that QNPE preserves the fast convergence rate of HPE, we develop a novel scheme for the update of $\mB_k$ to control the error caused by the linear approximation. 
As a result, unlike traditional quasi-Newton methods (such as BFGS and DFP) that update $\mB_k$ by mimicking some property of the Hessian such as the secant condition, our update rule is directly motivated by the convergence analysis of the HPE framework. %
Specifically, according to our analysis, it is sufficient to ensure that $\sum_{k} {1}/{\eta_k^2}<+\infty$ in order to guarantee a superlinear convergence rate for the QNPE method. As we discuss later, this sum can be explicitly bounded above by the cumulative loss $ \sum_k  \ell_k(\mB_k)$,
where  $\ell_k : \semiS^d \rightarrow \reals_+$ is a loss function that in some sense measures the approximation error. As a result, the update of $\mB_k$ boils down to running an online algorithm for solving  an \emph{online convex optimization problem} in the space of positive definite matrices with bounded eigenvalues. 

Finally, we address the challenge of computational efficiency by presenting a projection-free online learning algorithm for the update of $\mB_k$.
Note that most online learning algorithms for constrained problems are based on a projection oracle, but in our specific setting, such projection requires expensive eigendecomposition to ensure that the eigenvalues of $\mB_k$ are within a specific range. In contrast, our projection-free online learning algorithm is based on an approximate separation oracle (see Definition~\ref{def:gauge}) that can be efficiently constructed using matrix-vector products.

\section{Preliminaries}\label{sec:pre}

In this paper, we focus on the unconstrained minimization problem  
\vspace{-1mm}
\begin{equation}\label{eq:main_problem}
  \min_{\vx \in \mathbb{R}^d}\;\;\; f(\vx),\vspace{-1mm}
\end{equation}
where $f:\reals^d \rightarrow \reals$ is convex and twice differentiable and satisfies the following assumptions.

\vspace{-2mm}
\begin{assumption}\label{assum:smooth_SC} There exist positive constants $\mu$ and $L_1$ such that 
    $\mu \mI \preceq \nabla^2 f(\vx) \preceq L_1 \mI$ for any $\vx\in \reals^d$,
  where $\mI \in \reals^{d\times d}$ is the identity matrix. That is, $f$ is $\mu$-strongly convex and $L_1$-smooth.  
\end{assumption}

\vspace{-5mm}
\begin{assumption}\label{assum:Hessian_Lips} There exists $L_2>0$ such that   
  $\|\nabla^2 f(\vx) - \nabla^2 f(\vx^*)\|_{\op} \leq L_2 \|\vx-\vx^*\|_2$ for any $\vx\in \reals^d$, where $\vx^*$ is the optimal solution and $\|\mA\|_{\op} \triangleq \sup_{\vx:\|\vx\|_2 = 1} {\|\mA\vx\|_2}$. 
\end{assumption}
\vspace{-2mm}

We note that both assumptions are standard; Assumption~\ref{assum:smooth_SC} is common in the study of first-order methods, while Assumption~\ref{assum:Hessian_Lips} is also used in various papers on the superlinear convergence of classical quasi-Newton methods (e.g., see \cite{byrd1987global,jin2022non}). {For instance, both the regularized log-sum-exp function and the loss function of regularized logistic regression satisfy our assumptions (see \cite{rodomanov2021greedy}).}
Also, unless otherwise specified, throughout the paper we use $\|\cdot\|$ to denote the Euclidean norm.

\textbf{Hybrid Proximal Extragradient Framework.} 
To set the stage for our algorithm, we briefly recap the hybrid proximal extragradient (HPE) framework in \citep{solodov1999hybrid,monteiro2010complexity,monteiro2012iteration}.
When specialized to the minimization problem in \eqref{eq:main_problem}, it can be described by the following two steps:  
First, we perform an inexact proximal point update $\hat{\vx}_k \approx \vx_k-\eta_k \nabla f(\hat{\vx}_k)$ with step size  $\eta_k$. More precisely, for a given parameter $\sigma \in [0,1)$, we find $\hat{\vx}_k$ that satisfies
  \begin{equation}\label{eq:approx_pp}
    \|{\hat{\vx}_k} - \vx_k+{\eta_k} \nabla f( \hat{\vx}_{k})\| \leq \sigma \|\hat{\vx}_k-\vx_k\|. 
  \end{equation}
 Then, we perform an extragradient step and compute $\vx_{k+1}$ by 
  \begin{equation}\label{eq:extragradient}
    \vx_{k+1} = \vx_k - \eta_k \nabla f(\hat{\vx}_k).
  \end{equation}
While the proximal point update is only computed inexactly in \eqref{eq:approx_pp},  
\cite{monteiro2010complexity} proved that the HPE method can achieve a similar convergence guarantee as the proximal point method. Specifically, when $f$ is convex, it holds that  
$f(\bar{\vx}_{N-1})-f(\vx^*) \leq \frac{1}{2}\|\vx_0-\vx^*\|^2(\sum_{k=0}^{N-1}\eta_k)^{-1}$, where $\bar{\vx}_{N-1} \triangleq \nicefrac{\sum_{k=0}^{N-1}\eta_k\hat{\vx}_k}{\sum_{k=0}^{N-1} \eta_k} $ is the averaged iterate. 
It is worth noting that the HPE method is not directly implementable, but rather a useful conceptual tool, as we still need to specify how to find $\hat{\vx}$ satisfying the condition in~\eqref{eq:approx_pp}. One systematic approach is to approximate the gradient operator $\nabla f$ by a simpler local model $P(\vx;\vx_k)$, and then compute $\hat{\vx}_k$ by solving the equation
\vspace{-.3em}
\begin{equation}\label{eq:equation_model}
  \hat{\vx}_k - \vx_k + \eta_k P(\hat{\vx}_k;\vx_k) = 0.
\end{equation}
Furthermore, we can see that the condition in \eqref{eq:approx_pp} becomes 
\begin{equation}\label{eq:stepsize_model}
  \eta_k\|\nabla f(\hat{\vx}_k) - P(\hat{\vx}_k;\vx_k)\| \leq \sigma \|\hat{\vx}_k-\vx_k\|,
\end{equation}
which imposes an upper bound on the step size depending on the approximation error. 
For instance, if we take $P(\vx;\vx_k) = \nabla f(\vx_k)$, the update in \eqref{eq:equation_model} reads $\hat{\vx}_k = \vx_k-\eta_k \nabla f(\vx_k)$, leading to the classic extragradient method by \cite{Korpelevich1976}. 
If we use $P(\vx;\vx_k) = \nabla f(\vx_k)+ \nabla^2 f(\vx_k)(\vx-\vx_k)$,
we obtain the Newton proximal extragradient (NPE) method by \cite{monteiro2010complexity,monteiro2012iteration},
which has a faster convergence rate in the convex setting. %
However, the NPE method requires access to the objective function Hessian, which could be computationally costly. In this paper, we propose a quasi-Newton proximal extragradient method that only requires access to gradients. Surprisingly, our update rule for the Hessian approximation matrix does not follow traditional update rules such as the ones in BFGS or DFP, but is instead guided by an online learning approach,
where we aim to minimize the regret corresponding to certain approximation error. 
More details are in Section~\ref{sec:QNPE}.  
\section{Quasi-Newton Proximal Extragradient Method}\label{sec:QNPE}

  In this section, we propose the quasi-Newton proximal extragradient (QNPE) method.  
  An informal description is provided in Algorithm~\ref{alg:Full_EQN}.
  On a high level, our method falls into the HPE framework described in Section~\ref{sec:pre}. In particular, we choose the local model in \eqref{eq:equation_model} and \eqref{eq:stepsize_model} as $P(\vx;\vx_k) = \nabla f(\vx_k) + \mB_k(\vx-\vx_k)$, where $\mB_k\in \semiS^d$ is the Hessian approximation matrix. 
  Specifically, the update at the $k$-th iteration consists of three major stages, which we describe in the following paragraphs. 
  \begin{algorithm}[!t]\small
    \caption{Quasi-Newton Proximal Extragradient (QNPE) Method (informal)}\label{alg:Full_EQN}
      \begin{algorithmic}[1]
          \STATE \textbf{Input:} strong convexity parameter $\mu$, smoothness parameter $L_1$,
          line search parameters $\alpha_1\geq 0$ and $\alpha_2>0$ such that $\alpha_1+\alpha_2 < 1$,  and initial trial step size $\sigma_0>0$ 
          \STATE \textbf{Initialization:} initial point $\vx_0\in \mathbb{R}^d$ and initial Hessian approximation $\mB_0$ such that $\mu \mI \preceq \mB_0 \preceq L_1 \mI$
          \FOR{iteration $k=0,\ldots,N-1$}
          \STATE Let $\eta_k$ be the largest possible step size in $\{\sigma_k\beta^i:i\geq 0\}$ such that \label{line:LS_begin}  \tikzmark{top}
           \vspace{-3mm}
          \begin{align*}
           \vspace{-6mm}
            & \hspace{-7em}\hat{\vx}_{k} \approx_{\alpha_1} \vx_k - \eta_k (\mI+\eta_k \mB_k)^{-1} \nabla f(\vx_k), \qquad (\text{see Eq. \eqref{eq:inexact_linear_solver}})\tikzmark{right}\\ %
            &  \hspace{-7em}
            { \eta_k} \|\nabla f({\hat{\vx}_k})-\nabla f(\vx_k)-\mB_k({\hat{\vx}_k}-\vx_k)\| \leq {\alpha_2 \|{\hat{\vx}_k}-\vx_k\|}.  %
          \end{align*}%
          \vspace{-6mm}
          \STATE Set $\sigma_{k+1} \leftarrow \eta_{k}/\beta$ \label{line:intial_stepsize}\tikzmark{bottom}
          \STATE Update 
          $
            \vx_{k+1} \leftarrow \frac{1}{1+2\eta_k\mu}(\vx_k - \eta_k \nabla f(\hat{\vx}_k))+\frac{2\eta_k\mu}{1+2\eta_k\mu} \hat{\vx}_k
          $ \label{line:extragradient_step}
          \vspace{.3em}
          \IF[Line search accepted the initial trial step size ]{$\eta_k = \sigma_k$ \tikzmark{top2} } 
          \STATE Set $\mB_{k+1} \leftarrow \mB_k$ \label{line:Hessian_unchanged}
          \ELSE[Line search bactracked]
          \STATE Let $\tilde{\vx}_k$ be the last rejected iterate in the line search

          \STATE Set $\vy_k \leftarrow \nabla f(\tilde{\vx}_k) -\nabla f({\vx_k})$ and $\vs_k \leftarrow \tilde{\vx}_k-\vx_k$
          \STATE Define the loss function $\ell_k(\mB) = \frac{\|\vy_k-\mB\vs_k\|^2}{2\|\vs_k\|^2}$ 
          \STATE Feed $\ell_k(\mB)$ to an online learning algorithm and obtain $\mB_{k+1}$  \label{line:online_update} \phantom{\quad\quad}\tikzmark{right2}
          \ENDIF \tikzmark{bottom2}
          \ENDFOR
             \AddNote{top}{bottom}{right}{\color{comment}\textit{\quad Line search subroutine;\\
             \quad see Section~\ref{sec:line_search}}}
      \AddNote{top2}{bottom2}{right2}{\color{comment}\textit{\quad Hessian approximation\\
      \quad update subroutine; see\\ \quad Section~\ref{sec:no_regret}}}
      \end{algorithmic}
    \end{algorithm}

  In the \textbf{first stage}, given the Hessian approximation matrix $\mB_k$ and the current iterate $\vx_k$, we select the step size $\eta_k$ and the point $\hat{\vx}_k$ such that
  \vspace{-1mm}
   \begin{align}
      \|{\hat{\vx}_k}-\vx_k+{\eta_k}(\nabla f(\vx_k)+\mB_k({\hat{\vx}_k}-\vx_k))\| &\leq {\alpha_1} \|\hat{\vx}_k-\vx_k\|, \label{eq:inexact_linear_solver} \\
      { \eta_k} \|\nabla f(\hat{\vx}_k)-\nabla f(\vx_k)-\mB_k(\hat{\vx}_k-\vx_k)\| &\leq {\alpha_2 \|{\hat{\vx}_k}-\vx_k\|} \label{eq:step size_requirement}, 
        \vspace{-1mm}
    \end{align}
    where $\alpha_1\in [0,1)$ and $\alpha_2\in (0,1)$ are user-specified parameters with $\alpha_1+\alpha_2<1$. The first condition in \eqref{eq:inexact_linear_solver}
    requires $\hat{\vx}_k$ to be 
    an inexact solution of the linear system of equations $(\mI+{\eta_k}\mB_k)({\vx}-\vx_k) = -\eta_k\nabla f(\vx_k)$, where $\alpha_1$ controls the error of solving the linear system. In particular, when $\alpha_1=0$, it reduces to the update $\hat{\vx}_k=\vx_k -\eta_k(\mI+{\eta_k}\mB_k)^{-1}\nabla f(\vx_k)$ as in \eqref{eq:HPE_update}. %
     The second condition in \eqref{eq:step size_requirement} ensures that 
     the approximation error between the gradient $\nabla f(\hat{\vx}_k)$ and its quasi-Newton approximation $\nabla f(\vx_k)+\mB_k(\hat{\vx}_k-\vx_k)$ is sufficiently small. 
    To satisfy the conditions in  \eqref{eq:inexact_linear_solver} and \eqref{eq:step size_requirement} simultaneously, we need to determine the step size $\eta_k$ and the point $\hat{\vx}_k$ by a \textit{line search subroutine} (cf. Lines~\ref{line:LS_begin}-\ref{line:intial_stepsize} in Algorithm~\ref{alg:Full_EQN}).  Specifically, for a given parameter $\beta \in (0,1)$, we choose the largest admissible step size %
    from the set $\{\sigma_k \beta^i:i\geq 0\}$, where $\sigma_k = \eta_{k-1}/\beta$ for $k\geq 1$. This can be implemented by a  backtracking line search scheme and we present the details in Section~\ref{sec:line_search}. %
  
  In the \textbf{second stage}, we compute  $\vx_{k+1}$ using the gradient at  $\hat{\vx}_k$ (cf. Line~\ref{line:extragradient_step} in Algorithm~\ref{alg:Full_EQN}), but our update  is slightly different from the one in \eqref{eq:extragradient} as we focus on the strongly-convex setting, while the update in \eqref{eq:extragradient} is designed for the convex setting. More precisely, we compute $\vx_{k+1}$ according to
  \vspace{-1mm}
  \begin{equation}\label{eq:extragradient_mixing}
  \vx_{k+1} = \frac{1}{1+2\eta_k \mu}(\vx_k - \eta_k \nabla f(\hat{\vx}_k))+\frac{2\eta_k\mu}{1+2\eta_k\mu} \hat{\vx}_k, 
\end{equation}
where we choose the coefficients based on our analysis to obtain the best convergence rate.
Note that the above update in \eqref{eq:extragradient_mixing} reduces to \eqref{eq:extragradient} when $\mu=0$, and thus it can be viewed as an extension of HPE to the strongly-convex setting, which appears to be novel and of independent interest. 
In the \textbf{third stage}, we update the Hessian approximation matrix $\mB_{k}$. 
Here, we take a different approach from the classical quasi-Newton methods (such as BFGS and DFP) and let the convergence analysis guide our choice of $\mB_{k+1}$. As will be evident later, the convergence rate of Algorithm~\ref{alg:Full_EQN} is closely related to the cumulative loss $\sum_{k\in \mathcal{B}} \ell_k(\mB_k)$, where $\mathcal{B}$ denotes the set of iteration indices where the line search procedure backtracks. Here, the loss function is given by $\ell_k(\mB_k) \triangleq \frac{\|\vy_k-\mB_k\vs_k\|^2}{2\|\vs_k\|^2}$, where  $\vy_k = \nabla f(\tilde{\vx}_k)-\nabla f({\vx}_k)$, $\vs_k = \tilde{\vx}_k-\vx_k $, and $\tilde{\vx}_k$ is an auxiliary iterate returned by our line search scheme.  
   Thus, the update of the Hessian approximation matrix naturally fits into the framework of \textit{online learning}. 
    More precisely, 
   if the line search accepts the initial trial step size (i.e., $k\notin \mathcal{B}$), we keep the Hessian approximation matrix unchanged (cf. Line~\ref{line:Hessian_unchanged} in Algorithm~\ref{alg:Full_EQN}). Otherwise, we follow a tailored projection-free online learning algorithm in the space of matrices (cf. Line~\ref{line:online_update} in Algorithm~\ref{alg:Full_EQN}).  %
   The details of the update of $\mB_k$ are in Section~\ref{sec:no_regret}.

Finally, we provide a convergence guarantee for QNPE in Proposition~\ref{prop:IEQN}, which serves as a cornerstone for our convergence analysis. 
We note that the following result does not require additional conditions on $\mB_k$, other than the ones in  \eqref{eq:inexact_linear_solver} and \eqref{eq:step size_requirement} . 
The proof is available in Appendix~\ref{appen:IEG}. 
  \begin{proposition}\label{prop:IEQN}
    Let $\{\vx_k\}_{k\geq 0}$ be the iterates generated by \eqref{eq:inexact_linear_solver}, \eqref{eq:step size_requirement}, and \eqref{eq:extragradient_mixing} where $\alpha_1+\alpha_2 <1$. If $f$ is $\mu$-strongly convex, then  
    $\|\vx_{k+1}-\vx^*\|^2 \leq  \|\vx_k-\vx^*\|^2 (1+2\eta_k\mu)^{-1}$. 

  \end{proposition}
Proposition~\ref{prop:IEQN} highlights the pivotal role of  $\eta_k$ in the convergence rate: the larger the step size, the faster the convergence. On the other hand, $\eta_k$ is constrained by the condition in \eqref{eq:step size_requirement}, which, in turn, depends on the Hessian approximation matrix $\mB_k$. Thus, the central goal of our line search scheme and the Hessian approximation update is to make our step size $\eta_k$ as large as possible.

\subsection{Backtracking line search}\label{sec:line_search}
In this section, we describe a backtracking line search scheme for selecting the step size $\eta_k$ and the iterate $\hat{\vx}_k$ in the first stage of QNPE. For simplicity, we denote $\nabla f(\vx_k)$ by $\vg$ and drop the subscript $k$ in $\vx_k$ and $\mB_k$. %
Recall that at the $k$-th iteration, our goal is to find a pair $(\eta_{+}, \hat{\vx}_{+})$ such that 
\begin{align}
  \|\hat{\vx}_{+}-\vx+{\eta_{+}}(\vg+\mB({\hat{\vx}_{+}}-\vx))\| &\leq {\alpha_1} \|{\hat{\vx}_{+}}-\vx\|,
  \label{eq:x_plus_update} \\
  { \eta_{+}} \|\nabla f({\hat{\vx}_{+}})-\vg-\mB({\hat{\vx}_{+}}-\vx)\| &\leq \alpha_2 \|{\hat{\vx}_{+}}-\vx\| \label{eq:step size_condition}. 
\end{align}
As mentioned in the previous section, the condition in \eqref{eq:x_plus_update} can be satisfied if we solve the linear system $(\mI+\eta_{+}\mB)(\hat{\vx}_+-\vx) = -\eta_+\vg$ to a desired accuracy.
To formalize,  we let 
\begin{equation}\label{eq:linear_solver_update}
  \vs_{+} = \mathsf{LinearSolver}(\mI+\eta_{+}\mB, -\eta_{+}\vg; \alpha_1) \quad \text{and} \quad \hat{\vx}_{+} = \vx+\vs_{+},
\end{equation}
where the oracle $\mathsf{LinearSolver}$ is defined as follows. 
\begin{definition}\label{def:linear_solver}
  The oracle $\mathsf{LinearSolver}(\mA,\vb; \alpha)$ takes  a matrix $\mA \in \semiS^d$, a vector $\vb\in \reals^d$ and $\alpha \in (0,1)$ as input, and returns an approximate solution $\vs_{+}$ 
  satisfying $\|\mA\vs_{+}-\vb\| \leq \alpha \|\vs_{+}\|$. 
\end{definition}
By Definition~\ref{def:linear_solver}, the pair $(\eta_+,\hat{\vx}_+)$ is guaranteed to satisfy \eqref{eq:x_plus_update} when $\hat{\vx}_+$ is computed based on~ \eqref{eq:linear_solver_update}. 
 To implement the oracle $\mathsf{LinearSolver}(\mA,\vb; \alpha)$, the most direct way is to compute the exact solution $\vs_{+} = \mA^{-1}\vb$.  In Appendix~\ref{appen:CR}, we will describe a more efficient implementation via the conjugate residual method \citep{saad2003iterative}, which only requires computing matrix-vector products.
 
\begin{subroutine}[!t]\small
  \caption{Backtracking line search}\label{alg:ls}
  \begin{algorithmic}[1]
      \STATE \textbf{Input:} iterate $\vx \in \mathbb{R}^d$,  gradient $\vg\in \reals^d$,  Hessian approximation $\mB\in \semiS^{d}$, initial trial step size $\sigma>0$
      \STATE \textbf{Parameters:} line search parameters $\beta\in (0,1)$, $\alpha_1\geq 0$ and $\alpha_2>0$ such that $\alpha_1+\alpha_2<1$
      \STATE Set ${\eta}_{+} \leftarrow \sigma$, {$\vs_{+} \leftarrow \mathsf{LinearSolver}(\mI+\eta_{+}\mB, -\eta_{+}\vg; \alpha_1)$ and $\hat{\vx}_{+} \leftarrow \vx+\vs_{+}$}
      \WHILE{$\eta_{+} \|\nabla f(\hat{\vx}_{+})-\vg-\mB(\hat{\vx}_{+}-\vx)\|_2 > \alpha_2 \|\hat{\vx}_{+}-\vx\|_2$}
        \STATE Set $\tilde{\vx} \leftarrow \hat{\vx}_{+}$ %
          and $\eta_{+} \leftarrow \beta\eta_+ $
      \STATE Compute $\vs_{+} \leftarrow \mathsf{LinearSolver}(\mI+\eta_{+}\mB, -\eta_{+}\vg; \alpha_1)$ and $\hat{\vx}_{+} \leftarrow \vx+\vs_{+}$
      \ENDWHILE
      \IF{$\eta_+ = \sigma$}
      \STATE \textbf{Return} $\eta_{+}$ and  $\hat{\vx}_{+}$ \label{line:return_1}
      \ELSE 
      \STATE \textbf{Return} $\eta_{+}$, $\hat{\vx}_{+}$ and $\tilde{\vx}$ \label{line:return_2}
      \ENDIF

  \end{algorithmic}
\end{subroutine}

Now we are ready to describe our backtracking line search scheme in Subroutine~\ref{alg:ls} assuming access to the $\mathsf{LinearSolver}$ oracle.   
Specifically, given a user-defined parameter $\beta\in (0,1)$ and initial trial step size $\sigma>0$, we try the step sizes in $\{\sigma \beta^i: i\geq 0\}$ in decreasing order and compute $\hat{\vx}_+$ according to \eqref{eq:linear_solver_update}, until we find one pair $(\eta_+,\hat{\vx}_{+})$ that satisfies  \eqref{eq:step size_condition}. 
{Note that by standard arguments, the line search scheme is guaranteed to terminate in a finite number of steps.}
Since \eqref{eq:x_plus_update} already holds true by following the update rule in \eqref{eq:linear_solver_update}, the line search scheme will return a pair $(\eta_+,\hat{\vx}_{+})$ satisfying both conditions in \eqref{eq:x_plus_update} and \eqref{eq:step size_condition} .
Regarding the output, we distinguish two cases. If we pass the test in \eqref{eq:step size_condition} on our first attempt, we accept the initial step size $\sigma$ and the corresponding iterate $\hat{\vx}_+$ (cf. Line~\ref{line:return_1}). Otherwise, 
if $\sigma$ fails the test and we go through the backtracking procedure, along with the pair $(\eta_+,\hat{\vx}_{+})$, we also return an auxiliary iterate $\tilde{\vx}$, which is the last rejected point we compute from \eqref{eq:linear_solver_update} using the step size $\eta_+/\beta$ (cf. Line~\ref{line:return_2}). As we shall see in Lemma~\ref{lem:step size_lb}, the iterate $\tilde{\vx}$ is used to construct a lower bound on $\eta_+$, which will guide our update of the Hessian approximation matrix. 

For ease of notation, let $\mathcal{B}$ be the set of iteration indices where the line search scheme backtracks, i.e., $\mathcal{B}\triangleq \{k:\eta_k < \sigma_{k}\}$. For these iterations in $\mathcal{B}$, the next lemma provides a lower bound on the step size $\eta_k$ returned by our line search scheme, which will be the key to our convergence analysis and the update of the Hessian approximation matrices. The proof can be found in Appendix~\ref{appen:step size_lb}.

\vspace{-1mm}
\begin{lemma}\label{lem:step size_lb}
    For $k\notin \mathcal{B}$ we have $\eta_k = \sigma_k$, while for $k\in \mathcal{B}$ we have 
    \begin{equation}\label{eq:step size_lower_bound}
      \eta_k > \frac{\alpha_2 \beta\|{\tilde{\vx}_k}-\vx_k\|}{\|\nabla f({\tilde{\vx}_k})-\nabla f(\vx_k)-\mB_k({\tilde{\vx}_k}-\vx_k)\|} \quad \text{and} \quad \|{\tilde{\vx}_k}-\vx_k\| \leq \frac{1+\alpha_1}{\beta(1-\alpha_1)} \|{\hat{\vx}_k}-\vx_k\|.
    \end{equation}

\end{lemma}
\vspace{0mm}

In Lemma~\ref{lem:step size_lb}, we lower bound the step size $\eta_k$ in terms of the ratio between $\|\tilde{\vx}_k-\vx_k\|$ and the approximation error $\|\nabla f({\tilde{\vx}_k})-\nabla f(\vx_k)-\mB_k({\tilde{\vx}_k}-\vx_k)\|$. 
Hence, a better Hessian approximation matrix $\mB_k$ leads to a larger step size, which in turn implies faster convergence. 
Also, we note that the lower bound depends on the auxiliary iterate $\tilde{\vx}_k$ that is not accepted as the actual iterate. As such, we will use the second inequality in \eqref{eq:step size_lower_bound} to relate $\|\tilde{\vx}_k-\vx_k\|$ with $\|\hat{\vx}_k-\vx_k\|$. 
Finally, we remark that to fully characterize the computational cost of our method, we need to upper bound the total number of line search steps, each of which requires a call to $\mathsf{LinearSolver}$ and a call to the gradient oracle. This will be discussed later in Section~\ref{subsec:computational}. 

\vspace{-1mm}
\subsection{Hessian Approximation Update via Online Learning}\label{sec:no_regret}
In this section, we focus on the update rule for the Hessian approximation matrix $\mB_k$. Our goal is to develop a policy that leads to an explicit superlinear convergence rate for our proposed QNPE method. As mentioned earlier, 
our new policy %
differs greatly from the traditional quasi-Newton updates  
and is solely guided by the convergence analysis of our method. 

Our starting point is Proposition~\ref{prop:IEQN}, 
which characterizes the convergence rate of QNPE in terms of the step size $\eta_k$.
It implies that if we can show $\eta_k \rightarrow \infty$, then a superlinear convergence rate follows immediately. Specifically, by repeatedly applying the result of Proposition~\ref{prop:IEQN}, we obtain
\begin{equation}\label{eq:superlinear_jensen}
  \frac{\|\vx_{N}-\vx^*\|^2}{\|\vx_{0}-\vx^*\|^2} \leq \prod_{k=0}^{N-1} (1+2\eta_k\mu)^{-1} \leq \left(1+2\mu \sqrt{\frac{N}{\sum_{k=0}^{N-1} 1/\eta_k^2}}\right)^{-N},
\end{equation}
where the {last} inequality follows from Jensen's inequality {applied to the convex function $t\mapsto \log(1 + \frac{1}{t})$}. Hence, if we upper bound $\sum_{k=0}^{N-1} 1/\eta_k^2$ by a constant independent of $N$, it implies a global superlinear convergence rate of $\bigO({1}/\sqrt{N})^N$. Moreover, Lemma~\ref{lem:step size_lb} gives us the tool to control the step sizes and establish an upper bound on $\sum_{k=0}^{N-1} 1/\eta_k^2$, as shown in the following lemma. The proof is given in Appendix~\ref{appen:stepsize_bnd}.
 
\begin{lemma}\label{lem:stepsize_bnd}
$ \!\!\!$Let $\{\eta_k\}_{k=0}^{N-1}$ be the step sizes in Algorithm~\ref{alg:Full_EQN} using the line search in Subroutine~\ref{alg:ls}. Then,
 \vspace{-1mm}
  \begin{equation}\label{eq:goal}
    \sum_{k=0}^{N-1}\frac{1}{\eta_k^2} \leq \frac{1}{(1-\beta^2)\sigma_0^2}+ \frac{1}{(1-\beta^2)\alpha_2^2\beta^2}\sum_{k\in \mathcal{B}} \frac{\|\vy_k-\mB_k\vs_k\|^2}{\|\vs_k\|^2},
     \vspace{-1mm}
  \end{equation}
  where $\vy_k \triangleq \nabla f(\tilde{\vx}_k) -\nabla f({\vx_k})$ and $\vs_k \triangleq \tilde{\vx}_k-\vx_k$.
\end{lemma}

Our key idea is to interpret the right-hand side of \eqref{eq:goal} as the cumulative loss incurred by our choice of $\mB_k$, and to update the Hessian approximation matrix by an online learning algorithm. 
More formally, define the loss function at iteration $k$ as 
\begin{equation}\label{eq:loss_of_Hessian}
  \ell_k(\mB) \triangleq 
  \begin{cases}
    0, & \text{if } k\notin \mathcal{B}, \\
    \frac{\|\vy_k-\mB\vs_k\|^2}{2\|\vs_k\|^2}, & \text{otherwise}.
  \end{cases}
\end{equation}
Then the online learning protocol works as follows:
(i) At the $k$-th iteration, we choose $\mB_k\in \mathcal{Z}'$, where $\mathcal{Z}' \triangleq \{\mB\in \semiS^d: \frac{\mu}{2}\mI \preceq \mB \preceq (L_1+\frac{\mu}{2}) \mI\}$; 
 (ii) We receive the loss function $\ell_k(\mB)$ defined in \eqref{eq:loss_of_Hessian};  
 (iii) We update our Hessian approximation to $\mB_{k+1}$.  
Hence, minimizing the sum in \eqref{eq:goal} is equivalent to minimizing the cumulative loss $\sum_{k=0}^{N-1} \ell_k(\mB_k)$, which is exactly what online learning algorithms are designed for. In particular, we will show in Lemma~\ref{lem:small_loss} that the cumulative loss $\sum_{k=0}^{N-1} \ell_k(\mB_k)$ incurred by our online learning algorithm is comparable to $\sum_{k=0}^{N-1} \ell_k(\mH^*)$, where $\mH^* \triangleq \nabla^2 f(\vx^*)$ is the exact Hessian at the optimal solution $\vx^*$. 

\vspace{-1mm}
\begin{remark}\label{rem:constraint_on_Hessian}
  By Assumption~\ref{assum:smooth_SC}, we know that $\mu\mI \preceq \nabla^2 f(\vx) \preceq L_1\mI$. Thus, it is natural to restrict $\mB_k$ to the set $\mathcal{Z} \triangleq \{\mB\in \semiS^d: {\mu}\mI \preceq \mB \preceq L_1 \mI\}$. On the other hand, this constraint is by no means mandatory, and looser bounds on the eigenvalues of $\mB_k$ would also suffice for our analysis. Hence, we exploit this flexibility and allow our algorithm to pick $\mB_k$ from a larger set $\mathcal{Z}'$, as it is easier to enforce such a constraint. %
  We discuss this point in detail  in Section~\ref{subsec:projection_free}.
\end{remark}

\vspace{-5mm}
\begin{remark}
  Since we have $\ell_k(\mB)=0$ when $k\notin \mathcal{B}$,  we can simply keep $\mB_{k+1}$ unchanged for these iterations (cf. Line~\ref{line:Hessian_unchanged} in Algorithm~\ref{alg:Full_EQN}). With a slight abuse of notation, in the following, we relabel the indices in $\mathcal{B}$ by $t=0,1,\dots,T-1$ with $T \leq N$. 
\end{remark}

Now that we formulated the Hessian approximation update as an online learning problem, one can update $\mB_k$ by an online learning method, such as the projected online gradient descent~\citep{zinkevich2003online}. This approach would indeed serve our purpose and lead to an explicit superlinear convergence rate. However, in our setting, implementing  any projection-based online learning algorithm could be computationally expensive: the Euclidean projection onto the set $\mathcal{Z}'$ requires performing a full $d\times d$ matrix eigendecomposition, which typically incurs a complexity of $\mathcal{O}(d^3)$; please check Appendix~\ref{appen:ogd_regret} for more discussions.     
In the following, we instead build upon a projection-free online learning algorithm  proposed by~\cite{mhammedi2022efficient}. 

\subsubsection{Online Learning with an Approximate Separation Oracle}\label{subsubsec:projection_free_online_learning}
To better illustrate our key idea, we take a step back and consider a general online learning problem. %
For $T$ consecutive rounds $t=0,1,\dots,T-1$, a learner chooses an action $\vx_t \in \reals^n$ from an action set and then observes a loss function $\ell_t:\reals^n \rightarrow \reals$. The goal is to minimize the regret defined by $\mathrm{Reg}_T(\vx) \triangleq \sum_{t=0}^{T-1} \ell_t(\vx_t) - \sum_{t=0}^{T-1} \ell_t(\vx)$, which is the difference between the cumulative loss of the learner and that of a fixed competitor $\vx$.
This is a standard online learning problem, but with a slight modification: 
we restrict the competitor $\vx$ to be in a given competitor set~$\mathcal{C}$, while allow the learner to choose the action $\vx_t$ from a larger set $(1+\delta)\mathcal{C} \triangleq \{ (1+\delta)\vx: \vx\in \mathcal{C} \}$ for some given $\delta>0$. 
As mentioned in Remark~\ref{rem:constraint_on_Hessian}, this setup is more suitable for our Hessian approximation update framework, where the constraint on $\mB_t$ is more flexible %
(note that $\vx_t$ and $\vx$ correspond to $\mB_t$ and $\mH^*$, respectively).  
Finally, without loss of generality, we can assume that $0\in \calC$. %
We also assume that the convex set $\calC$ is bounded and contained in the Euclidean ball $\mathcal{B}_R(0)$ for some $R>0$.

To solve the above online learning problem, most existing algorithms require access to an oracle that computes the Euclidean projection on the action set. However, computing the projection is computationally costly in our setting (see Appendix~\ref{appen:ogd_regret}). Unlike these projection-based methods, here 
we rely on an approximate separation oracle defined below. 
As we discuss in Appendix~\ref{appen:SEP}, the following $\mathsf{SEP}$ oracle can be implemented much more efficiently than the Euclidean projection.

\begin{definition}\label{def:gauge}
The oracle $\mathsf{SEP}(\vw; \delta)$ takes $\vw\in \mathcal{B}_{R}(0)$ and $\delta>0$ as input and returns a scalar $\gamma>0$ and a vector $\vs\in \reals^n$ with one of the following possible outcomes:
 \vspace{-2mm}
  \begin{itemize}
    \item Case I: $\gamma \leq 1$ which implies that $\vw \in (1+\delta)\mathcal{C}$;
     \vspace{-2mm}
    \item Case II: $\gamma>1$ which implies that $\vw/\gamma \in (1+\delta)\mathcal{C} \ $ and $\ \langle \vs, \vw-\vx \rangle \geq {\gamma-1}$  $\quad\! \forall\vx\in \mathcal{C}$. 
  \end{itemize}
  \end{definition}
    \vspace{-0mm}
    In summary, the oracle $\mathsf{SEP}(\vw;\delta)$ has two possible outcomes: it either certifies that $\vw$ is approximately feasible, i.e., $\vw\in (1+\delta)\mathcal{C}$, or it produces a scaled version of $\vw$ that is in $(1+\delta)\mathcal{C}$ and gives a strict separating hyperplane between $\vw$ and the set $\mathcal{C}$.  

  \begin{remark}
   There are two main differences between Algorithm 1 in \citep{mhammedi2022efficient} and our presentation here. First, \cite{mhammedi2022efficient} considered a standard online learning setup where the action $\vx_t$ must be in the competitor set $\mathcal{C}$, while in our setting $\vx_t$ can be chosen from a larger set $(1+\delta)\mathcal{C}$. Second, their algorithm relied on an oracle that approximates the gauge function $\gamma_{\mathcal{C}}(\vw) \triangleq \inf\{\lambda \geq 0: \vw \in \lambda \mathcal{C}\}$ and its subgradient, which is further explicitly constructed using a membership oracle. Our oracle in Definition~\ref{def:gauge} is different but related, in the sense that its output $\gamma$ and $\vs$ may also be regarded as an approximation of the gauge function and its subgradient.    Moreover, we focus on the specific set used in our Hessian approximation update, and offer a more accustomed regret analysis and efficient construction of the oracle.  
  \end{remark}

    The key idea here is to introduce an auxiliary online learning problem on the larger set $\mathcal{B}_R(0)$ with surrogate loss functions $\tilde{\ell}_t (\vw) = \langle \tilde{\vg}_t,\vw\rangle$ for $0\leq t \leq T-1$, where $\tilde{\vg}_t$ is the surrogate gradient to be defined later. On a high level, we will run online projected gradient descent on this auxiliary problem to update the iterates $\{\vw_t\}_{t\geq 0}$ (note that the projection on $\mathcal{B}_R(0)$ is easy to compute), and then produce the actions $\{\vx_t\}_{t\geq 0}$ for the original problem by calling
    $\mathsf{SEP}(\vw_t;\delta)$ in Definition~\ref{def:gauge}. Specifically, given $\vw_t$ at round $t$, we let $\gamma_t>0$ and $\vs_t\in \reals^n$ be the output of $\mathsf{SEP}(\vw_t;\delta)$. 
    If $\gamma_t\leq 1$, we are in \textbf{Case I}, where we set $\vx_t = \vw_t$, compute $\vg_t = \nabla \ell_t(\vx_t)$, and define the surrogate gradient by $\tilde{\vg}_t = \vg_t$. 
    Otherwise, if $\gamma_t> 1$, we are in \textbf{Case II}, where we set $\vx_t = \vw_t/\gamma_t$, compute $\vg_t = \nabla \ell_t(\vx_t)$, and define the surrogate gradient by $\tilde{\vg}_t = \vg_t+\max\{0, -\langle \vg_t, \vx_t \rangle\} \vs_t$. Note that Definition~\ref{def:gauge} guarantees $\vx_t\in (1+\delta)\mathcal{C}$ in both cases.
    Finally, 
    we update $\vw_{t+1}$ using the standard online projected gradient descent with respect to the surrogate loss $\tilde{\ell}_t(\vw)$ and the set $\mathcal{B}_R(0)$:  
    \begin{equation*}
      \vw_{t+1}  = \Pi_{\mathcal{B}_R(0)}\left(\vw_t-\rho \tilde{\vg}_t\right) = \frac{R}{\max\{\|\vw_t-\rho \tilde{\vg}_t\|_2,R\}} (\vw_t-\rho \tilde{\vg}_t),
    \end{equation*}
    where $\rho$ is the step size. 
    To give some intuition, the surrogate loss functions $\{\tilde{\ell}_t (\vw)\}_{t=1}^T$ are constructed in such a way that the immediate regret $\tilde{\ell}_t{(\vw_t)}-\tilde{\ell}_t(\vx)$ serves as an upper bound on $\ell_t(\vx_t)-\ell_t(\vx)$ for any $\vx \in \mathcal{C}$. 
    Therefore, we can upper bound the regret of the original problem by that of the auxiliary problem, which 
    can be further bounded using the standard analysis for online projected gradient descent. The full algorithm is described in Algorithm~\ref{alg:projection_free_online_learning} in Appendix~\ref{appen:small_loss}.

\subsubsection{Projection-free Hessian Approximation Update}\label{subsec:projection_free}
\vspace{-1mm}

Now we are ready to describe our projection-free online learning algorithm for updating $\mB_k$, which is a special case of the algorithm described in the previous section. 
Recall that in our online learning problem in Section~\ref{sec:no_regret}, the competitor set is $\mathcal{Z} = \{\mB\in \semiS^d: \mu \mI \preceq \mB \preceq L_1 \mI\}$. Since the discussed projection-free scheme requires the competitor set $\mathcal{C}$ to contain the origin, we first translate and rescale $\mB$ via
the transform $\hat{\mB} \triangleq \frac{2}{L_1-\mu}\bigl(\mB-\frac{L_1+\mu}{2}\mI\bigr)$ and define $\mathcal{C} \triangleq \{\hat{\mB}\in \sS^d: -\mI \preceq \hat{\mB} \preceq \mI\} = \{\hat{\mB}\in \sS^d: \|\hat{\mB}\|_{\op}\leq 1\}$, which satisfies $0\in \mathcal{C}$ and $\mathcal{C} \subset \mathcal{B}_{\sqrt{d}}(0) = \{\mW \in \mathbb{S}^d:\|\mW\|_F \leq \sqrt{d}\}$.  
It can be verified that $\mB\! \in\! \mathcal{Z} \!\iff\! \hat{\mB} \in \mathcal{C}$, and also $\mB \in \mathcal{Z}' \!\iff \!\hat{\mB} \in (1+\delta)\mathcal{C}$ with $\delta= \mu/(L_1-\mu)$. 

\begin{subroutine}[!t]\small
  \caption{Online Learning Guided Hessian Approximation Update}\label{alg:hessian_approx}
  \begin{algorithmic}[1]
      \STATE \textbf{Input:} Initial matrix $\mB_0\in \sS^d$ s.t. $\mu \mI \preceq \mB_0 \preceq L_1\mI$, step size $\rho>0$, $\delta>0$, $\{q_t\}_{t=1}^{T-1}$
      \STATE \textbf{Initialize:} set $\mW_0 \leftarrow \frac{2}{L_1-\mu}(\mB_0-\frac{L_1+\mu}{2}\mI)$, $\mG_0 \leftarrow \frac{2}{L_1+\mu}\nabla \ell_0(\mB_0)$ and $\tilde{\mG}_0 \leftarrow \mG_0$
      \STATE {Update $\mW_{1} \leftarrow \frac{\sqrt{d}}{\max\{\sqrt{d},\|\mW_0 - \rho \tilde{\mG}_0\|_F\}}(\mW_0 - \rho \tilde{\mG}_0)$}
      \FOR{$t=1,\dots,T-1$}
      \STATE Query the oracle $(\gamma_t,\mS_t) \leftarrow \mathsf{ExtEvec}(\mW_t;\delta, q_t)$ %
      \IF[Case I]{$\gamma_t \leq 1$}
        \STATE Set $\hat{\mB}_t \leftarrow \mW_t$ and $\mB_t \leftarrow \frac{L_1-\mu}{2}\hat{\mB}_t+\frac{L_1+\mu}{2}\mI$
        \STATE Set $\mG_t \leftarrow \frac{2}{L_1-\mu}\nabla \ell_t(\mB_t)$ and $\tilde{\mG}_t \leftarrow \mG_t$
      \ELSE[Case II]
        \STATE Set $\hat{\mB}_t \leftarrow \mW_t/\gamma_t$ and $\mB_t \leftarrow \frac{L_1-\mu}{2}\hat{\mB}_t+\frac{L_1+\mu}{2}\mI$ 
        \STATE Set $\mG_t \leftarrow \frac{2}{L_1-\mu}\nabla \ell_t(\mB_t)$ and $\tilde{\mG}_t \leftarrow \mG_t+\max\{0,-\langle \mG_t, \mB_t \rangle\} \mS_t$
      \ENDIF
      \STATE %
        Update 
        $\mW_{t+1} \leftarrow \frac{\sqrt{d}}{\max\{\sqrt{d},\|\mW_t - \rho \tilde{\mG}_t\|_F\}}(\mW_t - \rho \tilde{\mG}_t)$ \COMMENT{Euclidean projection onto $\mathcal{B}_{\sqrt{d}}(0)$}
      \ENDFOR
  \end{algorithmic}
\end{subroutine}

The only remaining question is how we can build the $\mathsf{SEP}$ oracle in Definition~\ref{def:gauge} for our specific set $\mathcal{C}$. 
To begin with, we observe that this is closely related to 
computing the extreme eigenvalues and the associated eigenvectors of a given matrix $\mW$. In fact, let $\lambda_{\max}$ and $\vv_{\max}\in \reals^d$  be the largest magnitude eigenvalue of $\mW$ and its associated unit eigenvector, respectively. Since $\|\mW\|_\op = |\lambda_{\max}|$, it is easy to see that: (i) If $|\lambda_{\max}| \leq 1$, then $\mW \in \mathcal{C}$; (ii) Otherwise, if $|\lambda_{\max}| > 1$, then we can let $\gamma = |\lambda_{\max}|$, which satisfies $\mW/\gamma \in \mathcal{C}$, and $\mS = \sign(\lambda_{\max}) \vv_{\max}\vv_{\max}^\top \in \mathbb{S}^d$, which defines a separating hyperplane between $\mW$ and $\mathcal{C}$. Indeed, note that we have $\langle \mS,\mW\rangle = |\lambda_{\max}| = \gamma$ and $\langle \mS,\hat{\mB}\rangle \leq |\vv_{\max}^\top \hat{\mB}\vv_{\max}| \leq 1$ for any $\hat{\mB}\in \mathcal{C}$, which implies $\langle \mS,\mW-\hat{\mB}\rangle \geq \gamma -1$. Hence, we can build the separation oracle in Definition~\ref{def:gauge} if we compute $\lambda_{\max}$ and $\vv_{\max}$ for the given matrix $\mW$.  

However, the exact computation of $\lambda_{\max}$ and $\vv_{\max}$ could be costly. Thus, we propose to compute the extreme eigenvalues and the corresponding eigenvectors inexactly by the randomized Lanczos method~\citep{kuczynski1992estimating}, which leads to the randomized oracle $\mathsf{ExtEvec}$ defined below. We defer the specific implementation details of $\mathsf{ExtEvec}$ to Section~\ref{appen:SEP}.

\vspace{-1mm}
  \begin{definition}\label{def:extevec}
    The oracle $\mathsf{ExtEvec}(\mW;\delta,q)$ takes $\mW \in \mathbb{S}^d$, $\delta>0$, and $q\in (0,1)$ as input and returns a scalar $\gamma>0$ and a matrix $\mS\in \mathbb{S}^d$ with one of the following possible outcomes:
    \vspace{-2mm}
    \begin{itemize}
      \item Case I: $\gamma \leq 1$, which implies that, with probability at least $1-q$, $\|\mW\|_{\op} \leq 1+\delta$;
      \vspace{-2mm} 
      \item Case II: $\gamma>1$, which implies that, with probability at least $1-q$, $\|\mW/\gamma\|_{\op} \leq 1+\delta$, $\|\mS\|_F = 1$ and $\langle \mS,\mW-\hat{\mB}\rangle \geq \gamma -1$ for any $\hat{\mB}$ such that $\|\hat{\mB}\|_{\op} \leq 1$.
    \end{itemize}
  \end{definition}

  Note that $\mathsf{ExtEvec}$ is an approximate separation oracle for the set $\mathcal{C}$ in the sense of Definition~\ref{def:gauge} (with success probability at least $1-q$), and it also guarantees that $\|\mS\|_F = 1$ in Case II. Equipped with this oracle, we describe the complete Hessian approximation update in Subroutine~\ref{alg:hessian_approx}.

\section{Complexity Analysis of QNPE}\label{sec:convergence}
By now, we have fully described our QNPE method in Algorithm~\ref{alg:Full_EQN}, where we select the  $\eta_k$  by Subroutine~\ref{alg:ls} and update the Hessian approximation matrix $\mB_k$ by Subroutine~\ref{alg:hessian_approx}.  
In the following, we shall establish the convergence rate and characterize the computational cost of QNPE.

Next, we state our main convergence result. Our results hold for any $\alpha_1, \alpha_2 \in (0,\frac{1}{2})$ and $\beta\in(0,1)$, but to simplify our expressions we report the results for $\alpha_1 = \alpha_2 = \frac{1}{4}$ and $\beta = \frac{1}{2}$.
\begin{theorem}[Main Theorem]\label{thm:main}
  Let $\{\vx_k\}_{k\geq 0}$ be the iterates generated by Algorithm~\ref{alg:Full_EQN} using the line search scheme in Subroutine~\ref{alg:ls}, where $\alpha_1 = \alpha_2 = \frac{1}{4}$, $\beta = \frac{1}{2}$, and $\sigma_0 \geq  \alpha_2\beta/L_1$, and the Hessian approximation update in Subroutine~\ref{alg:hessian_approx}, where $\rho = \frac{1}{18}$, $\delta = \min\{\frac{\mu}{L_1-\mu},1\}$, and $q_t = \nicefrac{p}{2.5(t+1)\log^2(t+1)}$ for $t \geq 1$. Then with probability at least $1-p$, the following statements hold:  
  \begin{enumerate}[(a)]
  \vspace{-1mm}
    \item (Linear convergence) For any $k\geq 0$, we have
        $ \frac{\|\vx_{k+1}-\vx^*\|^2}{\|\vx_k-\vx^*\|^2 }\leq  \left(1+\frac{\mu}{4L_1}\right)^{-1}$.
        \vspace{-1mm}
    \item (Superlinear convergence) We have $\lim_{k\rightarrow \infty}\frac{\|\vx_{k+1}-\vx^*\|^2}{\|\vx_k-\vx^*\|^2 } = 0$. Furthermore, for any $k\geq 0$, 
    \vspace{-1mm}
    \begin{equation*}
      \!\!\!\!\!\frac{\|\vx_k-\vx^*\|^2}{\|\vx_0-\vx^*\|^2} \leq   \Biggl(1+ \frac{\sqrt{3}}{8}\mu\sqrt{{\frac{k}{L_1^2+{36}\|\mB_0-\nabla^2 f(\vx^*)\|^2_F + \left(27+
      \!\frac{16L_1}{\mu}\right)\!L_2^2\|\vx_0-\vx^*\|^2}}}\Biggr)^{\!\!-k}\!\!\!\!.
    \end{equation*}
  \end{enumerate}
\end{theorem}

 \vspace{-1mm}
 \noindent \textbf{Proof sketch.} 
 By using a simple union bound, we can show that 
 the  $\mathsf{ExtEvec}$ oracle in Subroutine~\ref{alg:hessian_approx} is successful in all  rounds with probability at least $1-p$. 
 Thus, throughout the proof, we assume that every call of $\mathsf{ExtEvec}$ is successful.  
We first prove the linear convergence rate in (a). As we discussed in Section~\ref{subsec:projection_free}, Subroutine~\ref{alg:hessian_approx} ensures that $\frac{\mu}{2}\mI \preceq \mB_k \preceq L_1+\frac{\mu}{2}\mI$ for any $k\geq 0$. Combining this with Lemma~\ref{lem:step size_lb}, we obtain the following universal lower bound on the step size $\eta_k$.  
\begin{lemma}\label{lem:stepsize_const_bound}
    For any $k\geq 0$, we have $\eta_k \geq 1/(8L_1)$. 
\end{lemma}
In light of Lemma~\ref{lem:stepsize_const_bound}, the linear convergence result in (a) now follows directly from Proposition~\ref{prop:IEQN}. 

Next, we prove the superlinear convergence rate in (b) by considering the following steps. 

\vspace{0em}\noindent\textbf{Step 1:} Using regret analysis, we bound the cumulative loss $\sum_{t=0}^{T-1}\ell_t(\mB_t)$ incurred by our online learning algorithm in Subroutine~\ref{alg:hessian_approx}. In particular, by exploiting the smooth property of the loss function $\ell_t$, we prove a small-loss bound in the following lemma, where the cumulative loss of the learner is bounded by that of a fixed action in the competitor set \citep{srebro2010smoothness}. 
  \begin{lemma}
  \label{lem:small_loss}
  For any ${\mH}\in \mathcal{{Z}}$, we have 
    $\sum_{t=0}^{T-1} \ell_t({\mB}_t)  \leq 18\|\mB_0-{\mH}\|_F^2+ 2 \sum_{t=0}^{T-1} \ell_t({\mH})$.
\end{lemma}
Note that in Lemma~\ref{lem:small_loss}, we have the freedom to choose any competitor $\mH$ in the set $\mathcal{Z}$. To further obtain an explicit bound, a natural choice would be $\mH^* \triangleq \nabla^2 f(\vx^*)$, which leads to our next step.

\vspace{0em}\noindent \textbf{Step 2:} We upper bound the cumulative loss $\sum_{t=0}^{T-1} \ell_t({\mH}^*)$ in the following lemma. The proof relies crucially on Assumption~\ref{assum:Hessian_Lips} as well as the linear convergence result we proved in (a). 
  \begin{lemma}\label{lem:comparator}
    We have 
    $\sum_{t=0}^{T-1} \ell_t(\mH^*) \leq \left(\frac{27}{4}+\frac{4L_1}{ \mu}\right)L_2^2\|\vx_0-\vx^*\|^2$.
  \end{lemma}

\noindent\textbf{Step 3}: Combining Lemma~\ref{lem:comparator} and  Lemma~\ref{lem:small_loss},  we obtain a constant upper bound on the cumulative loss $ \sum_{t=0}^{T-1} \ell_t({\mB}_t)$. By Lemma~\ref{lem:stepsize_bnd}, this further implies an upper bound on $\sum_{k=0}^{N-1} 1/\eta_k^2$, which leads to the superlinear convergence result in (b) by Proposition~\ref{prop:IEQN} and the observation in \eqref{eq:superlinear_jensen}. \jmlrQED

\noindent \textbf{Discussions.}
To begin with, Part (a) of Theorem~\ref{thm:main} guarantees that QNPE converges linearly and is at least as fast as gradient descent. Moreover, in Part (b) we prove Q-superlinear convergence of QNPE, where the explicit global superlinear rate is faster than the linear rate for sufficiently large~$k$.  Specifically, if we define $N_\mathrm{tr} \triangleq \frac{4}{3}+\frac{48}{L_1^2}\|\mB_0-\nabla^2 f(\vx^*)\|^2_F + \left(\frac{36}{L_1^2}+\frac{64}{3\mu L_1}\right)\!L_2^2\|\vx_0-\vx^*\|^2$, then the superlinear rate can be written as  $(1+\frac{\mu}{4L_1}\sqrt{\frac{k}{N_\mathrm{tr}}})^{-k}$, which is superior to the linear rate when $k \geq N_{\mathrm{tr}}$. 
Moreover, we can also derive an explicit complexity bound from Theorem~\ref{thm:main}. Let $N_{\epsilon}$ denote the number of iterations required by QNPE to achieve $\epsilon$-accurate solution, i.e., $\|\vx_k-\vx^*\|^2 \leq \epsilon$. As we show in Appendix~\ref{appen:complexity},
if the error tolerance $\epsilon$ is in the regime where ${\epsilon} > \exp(-\frac{\mu}{L_1} N_\mathrm{tr})$, the linear rate in Part (a) is faster and we have $N_{\epsilon}= \bigO(\frac{L_1}{\mu}\log\frac{1}{\epsilon})$. Otherwise, if ${\epsilon} < \exp(-\frac{\mu}{L_1} N_\mathrm{tr})$, the superlinear rate in Part (b) excels and we have $N_{\epsilon}= \bigO\Bigl( \Bigl[{ \log \Bigl(1+\frac{\mu}{L_1}\left({\frac{L_1}{ N_\mathrm{tr}\mu}\log\frac{1}{\epsilon}}\right)^{\nicefrac{1}{3}}\Bigr)}\Bigr]^{-1}\log\frac{1}{\epsilon}\Bigr)$.  

A couple of additional remarks about Theorem~\ref{thm:main} follow. First, the expression $\|\mB_0-\nabla^2 f(\vx^*)\|^2_F$ is bounded above by $L_1^2 d$ in the worst-case, showing that at worst $N_\mathrm{tr}$ scales linearly with the dimension $d$. On the other hand, $N_\mathrm{tr}$ could be much smaller if the initial Hessian approximation matrix $\mB_0$ is close to $\nabla^2 f(\vx^*)$. Second, Theorem~\ref{thm:main} provides a global result, as both bounds hold for any initial point $\vx_0$ and any initial Hessian approximation $\mB_0$.  %
On the contrary, the existing non-asymptotic results on quasi-Newton methods in \citep{rodomanov2021new,jin2022non} require special initialization for $\mB_0$ and closeness of $\vx_0$ to the optimal solution $\vx^*$.

\vspace{-2mm}
\subsection{Characterizing the Computational Cost}\label{subsec:computational}
As for most optimization algorithms, we measure the computational cost of our QNPE method in two aspects: the number of gradient evaluations and the number of matrix-vector product evaluations. 
In particular, each backtracking step of the line search scheme in Subroutine~\ref{alg:ls} requires one call to the gradient oracle, while the implementation of $\mathsf{LinearSolver}$ in Definition~\ref{def:linear_solver} and $\mathsf{ExtEvec}$ in Definition~\ref{def:extevec} requires multiple matrix-vector products. Due to space limitations, we defer the details to Appendix~\ref{appen:computational_cost} and summarize the complexity results in the following theorem.

\begin{theorem}\label{thm:computational_cost}
  Let $N_\epsilon$ denote the minimum number of iterations required by Algorithm~\ref{alg:Full_EQN} to find an $\epsilon$-accurate solution according to Theorem~\ref{thm:main}. Then, with probability at least $1-p$: 
  \begin{enumerate}[(a)]
  \vspace{-2mm}
    \item The total number of gradient evaluations is bounded by $3N_\epsilon+\log_{1/\beta}(4\sigma_0 L_1)$. 

     \vspace{-2mm}
    \item The total number of matrix-vector products in $\mathsf{ExtEvec}$ and $\mathsf{LinearSolver}$ are bounded by $\bigO \left(N_\epsilon \sqrt{\frac{L_1}{\mu}} \log \left(\frac{dN_{\epsilon}^2}{p^2}\right)\right)$ and $\bigO\left(N_{\epsilon}\sqrt{\frac{L_1}{\mu}}\log \left(\frac{L_1 \|\vx_0-\vx^*\|^2}{\mu \epsilon}\right)\right)$, respectively.  
  \end{enumerate}
\end{theorem}

As a direct corollary, on average QNPE requires at most $3$ gradient evaluations per iteration if we set $\sigma_0 = 1/(4L_1)$.  
  Moreover, by summing the complexity of both $\mathsf{ExtEvec}$ and $\mathsf{LinearSolver}$, we can bound the total number of matrix-vector products by $\bigO(N_{\epsilon}\sqrt{\frac{L_1}{\mu}}\log\frac{L_1 N_\epsilon^2d}{\mu\epsilon})$. 

\section{Numerical Experiments}

To verify our theoretical findings, we consider a regularized logistic regression problem on a synthetic dataset $\{(\va_i, y_i)\}_{i=1}^n$, where $\va_i \in \mathbb{R}^d$ is the $i$-th feature vector and $y_i\in \{+1,-1\}$ is the $i$-th binary label (details on the dataset can be found in Appendix~\ref{appen:experiments}). It can be formulated as the following optimization problem
\begin{equation*}
    \min_{\vx \in \reals^d} \;\;\; f(\vx) = \frac{1}{n} \sum_{i=1}^n \log(1+e^{-y_i \langle \va_i, \vx \rangle}) + \frac{\mu}{2}\|\vx\|^2,
\end{equation*}
where $\mu$ is the regularization parameter. In our experiment, we set $d = 150$, $n = 2000$ and $\mu = 0.005$, with the condition number $L_1/\mu$ estimated to be $7600$.

We implemented our proposed method QNPE following Algorithm~\ref{alg:Full_EQN}, where we select the step size $\eta_k$ by Subroutine~\ref{alg:ls} and update the Hessian approximation matrix $\mB_k$ by Subroutine~\ref{alg:hessian_approx}. Moreover, the $\mathsf{LinearSolver}$ oracle is implemented using the conjugate residual method (see Subroutine~\ref{alg:CRM}), while the $\mathsf{ExtEvec}$ oracle is implemented using MATLAB's eig function (we can afford full eigendecomposition since the dimension $d$ is relatively small in our test problem). 
For comparison, we also tested 
gradient descent (GD) and the classical BFGS quasi-Newton method, where we use line search to obtain their best performance \citep{Nocedal2006}.

{
From Figure~\ref{fig:logistic}(a), 
we observe that GD converges to the optimal solution at a slow linear rate, while QNPE and BFGS can achieve a high accuracy in much fewer iterations.  
We also illustrated our theoretical bound of $(1+c\sqrt{k})^{-k}$ with a manually tuned parameter $c$, which matches well with the empirical performance of QNPE. Due to the use of line search, in Figure~\ref{fig:logistic}(b) we also compare these algorithms in terms of the number of gradient evaluations. Note that the line search scheme in GD only queries the function value at the new point, and thus it requires exactly one gradient evaluation per iteration. As a result, while QNPE still converges faster than GD, the relative performance gap becomes smaller. 
On the other hand, we remark that the number of gradient evaluations per iteration for QNPE is still small as guaranteed by Theorem~\ref{thm:computational_cost}. 
Indeed, as shown in the histogram in Figure~\ref{fig:logistic}(c),  
most of the iterations evaluate 2 or 3 gradients and the average is no more than 3, which we observe consistently across different settings.  Finally, we note that BFGS with line search outperforms all the other considered methods in our experiments, despite the fact that its finite-time complexity bound is still lacking. %
Hence, establishing a global non-asymptotic convergence rate for BFGS is an interesting open problem to explore.
}
\begin{figure}[t]
\centering
\scriptsize
\subfigure[Convergence by iteration][b]{
\includegraphics[width=0.31\linewidth]{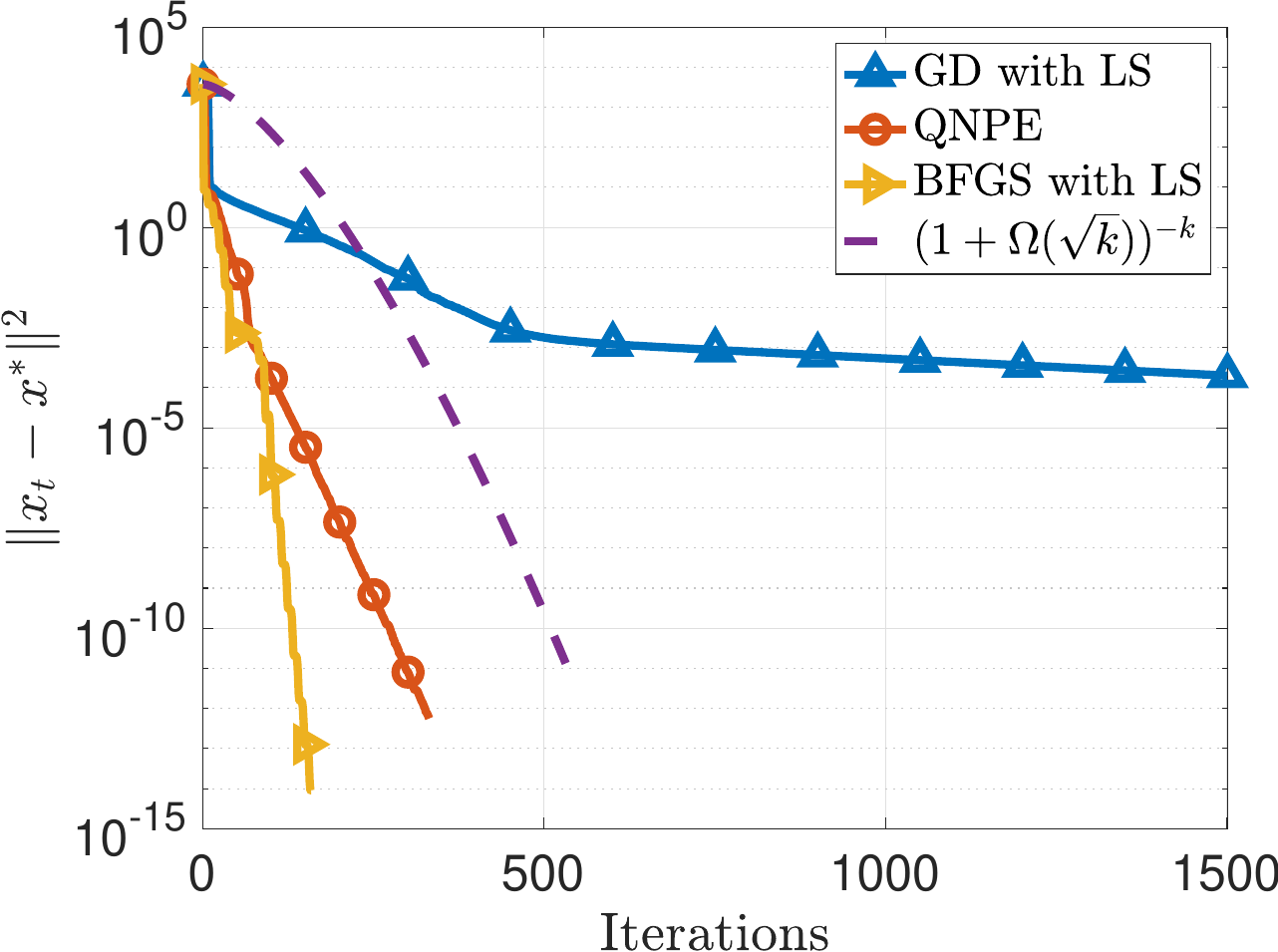}
}\,
\subfigure[Convergence by gradient evaluations][b]{
\includegraphics[width=0.31\linewidth]{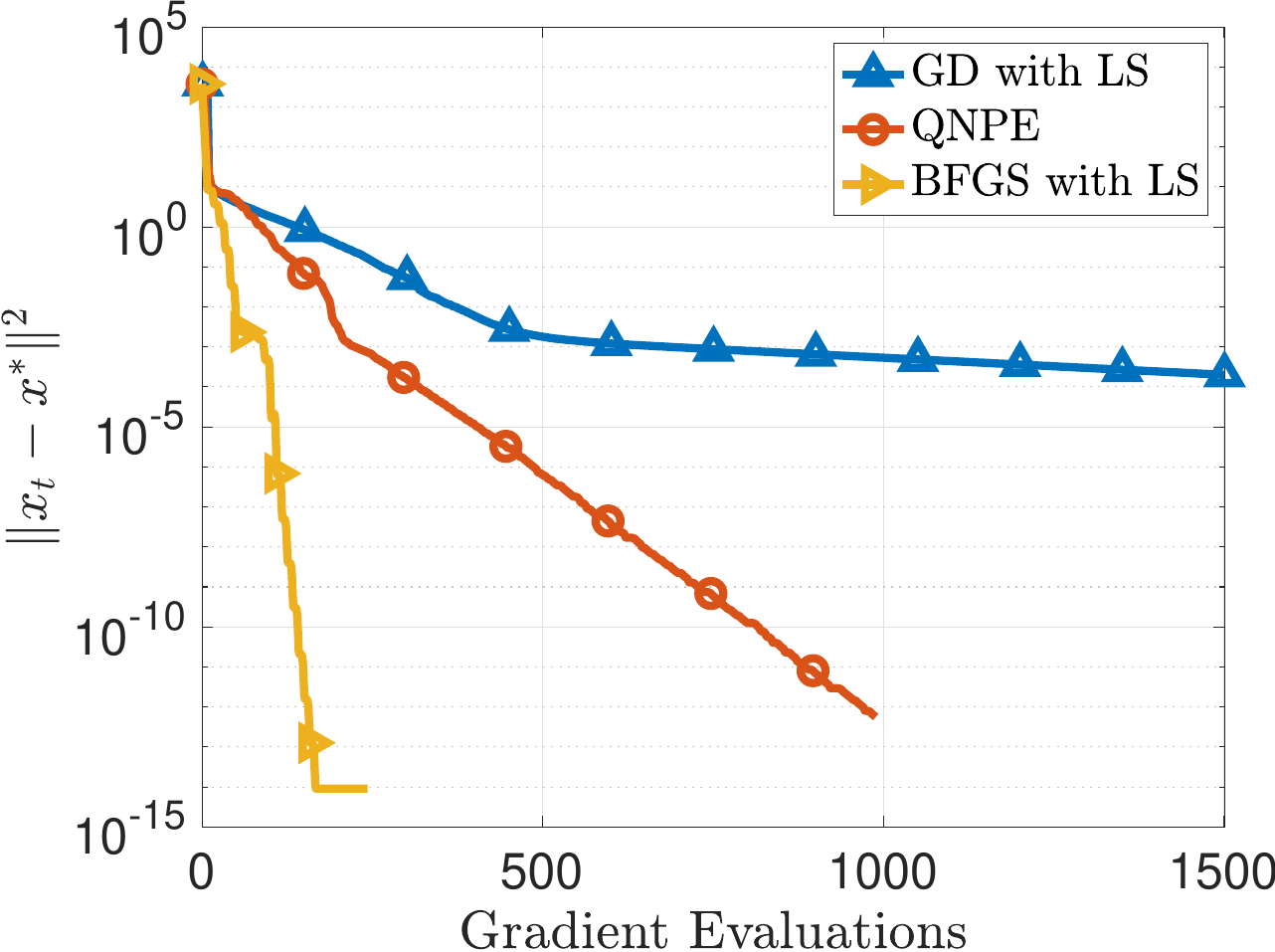}
}\,
\subfigure[Histogram of gradient evaluations][b]{
\includegraphics[width=0.32\linewidth]{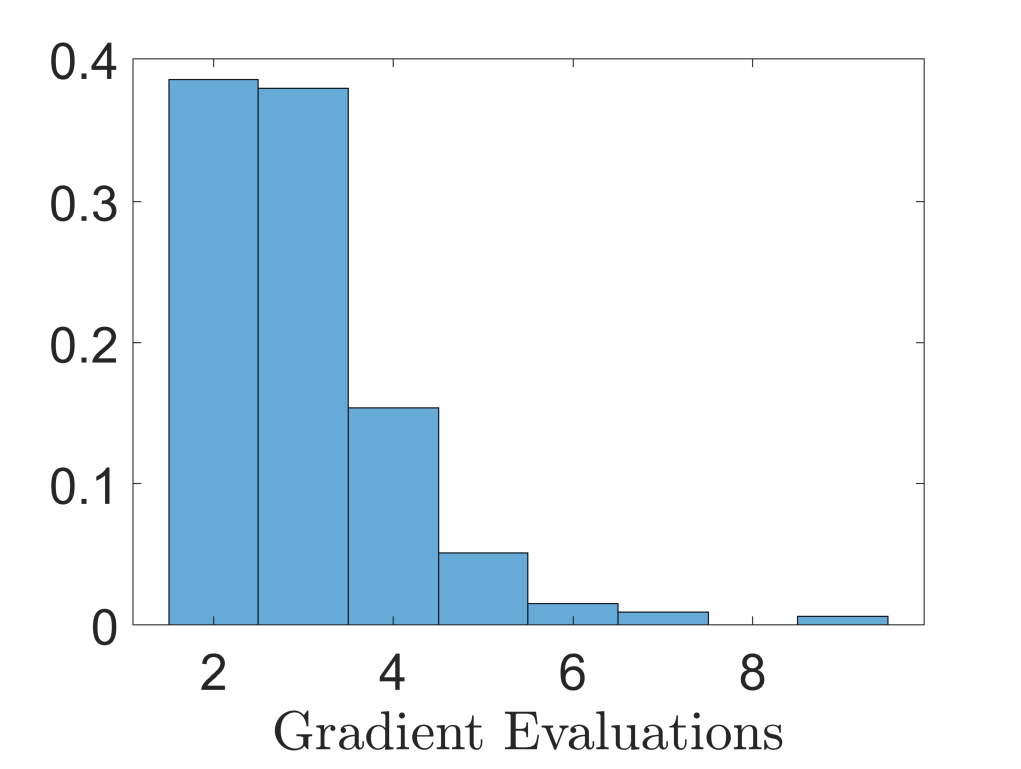}
}
\caption{Numerical results for a regularized logistic regression problem.}\label{fig:logistic}
\end{figure}
\vspace{-1mm}
\section{Conclusion}
\vspace{-1mm}
We proposed the quasi-Newton proximal extragradient (QNPE) method for unconstrained minimization problems. We showed that QNPE converges at an explicit non-asymptotic superlinear rate of $(1+\Omega(\sqrt{k}))^{-k}$. 
Moreover, if $N_{\epsilon}$ denotes the number of iterations to find an $\epsilon$-accurate solution, we showed that the number of gradient evaluations is bounded by $3N_{\epsilon}$, while the number of matrix-vector product evaluations is bounded by $\bigO(N_{\epsilon}\sqrt{\frac{L_1}{\mu}}\log\frac{L_1 N_\epsilon^2d}{\mu\epsilon})$. To the best of our knowledge, this is the first quasi-Newton method with an explicit global superlinear convergence rate.
\acks{This work is supported in part by NSF Grants 2007668, 2019844, and  2112471,  ARO  Grant  W911NF2110226,  the  Machine  Learning  Lab  (MLL)  at  UT  Austin, and the Wireless Networking and Communications Group (WNCG) Industrial Affiliates Program. 
The authors would also like to thank Anton Rodomanov and {the anonymous reviewers} for their comments on the first draft of the paper.}

\bibliography{quasi-newton}

\newpage

\appendix

\section{Missing Proofs in Section~\ref{sec:QNPE}}
\subsection{Proof of Proposition~\ref{prop:IEQN}}\label{appen:IEG}

In this section, we provide the proof of Proposition~\ref{prop:IEQN}. We also prove an additional result in \eqref{eq:sum_of_square}, which will be useful later in the proof of Lemma~\ref{lem:comparator}. 
\begingroup
\def\thetheorem{\ref{prop:IEQN}}
\begin{proposition}
  Let $\{\vx_k\}_{k\geq 0}$ be the iterates generated by \eqref{eq:inexact_linear_solver}, \eqref{eq:step size_requirement}, and \eqref{eq:extragradient_mixing} where $\alpha_1+\alpha_2 <1$. If $f$ is $\mu$-strongly convex, then we have 
  \begin{equation}\label{eq:IEQN_contraction}
    \|\vx_{k+1}-\vx^*\|^2 \leq  \|\vx_k-\vx^*\|^2 (1+2\eta_k\mu)^{-1}.
  \end{equation}  
 Moreover, we have 
  \begin{equation}\label{eq:sum_of_square}
    \sum_{k=0}^{N-1} \|\hat{\vx}_{k}-\vx_k\|^2 \leq \frac{1}{1-\alpha_1-\alpha_2}\|\vx_0-\vx^*\|^2.
  \end{equation}
\end{proposition}
\addtocounter{theorem}{-1}
\endgroup 
\begin{proof}
To simplify the notation, let $\alpha = \alpha_1+\alpha_2\in (0,1)$. For any $\vx\in \reals^d$, we can write 
\begin{equation}\label{eq:inexact_pp_decomp}
  \eta_k \langle \nabla f(\hat{\vx}_k), \hat{\vx}_k-\vx \rangle =  \langle \hat{\vx}_k-\vx_k+\eta_k\nabla f(\hat{\vx}_k), \hat{\vx}_k-\vx \rangle
  +\langle \vx_k-\hat{\vx}_k, \hat{\vx}_k-\vx\rangle. %
\end{equation}
To begin with, by using the triangle inequality and the conditions in \eqref{eq:inexact_linear_solver} and \eqref{eq:step size_requirement}, we observe that 
      \begin{align}
      &\phantom{{}={}}\|{\hat{\vx}_k} - \vx_k+{\eta_k} \nabla f( \hat{\vx}_{k})\| %
      \nonumber\\
      &=  \|{\hat{\vx}_k} - \vx_k+\eta_k(\nabla f(\vx_k)+\mB_k({\hat{\vx}_k}-\vx_k))+{\eta_k} \nabla f( \hat{\vx}_{k})-\eta_k(\nabla f(\vx_k)+\mB_k({\hat{\vx}_k}-\vx_k))\| \nonumber\\
      &\leq \|{\hat{\vx}_k} - \vx_k+\eta_k(\nabla f(\vx_k)+\mB_k({\hat{\vx}_k}-\vx_k))\|+\eta_k\|\nabla f( \hat{\vx}_{k})-\nabla  f(\vx_k) - \mB_k(\hat{\vx}_k-\vx_k)\| \nonumber\\
      &\leq (\alpha_1+\alpha_2) \|{ \hat{\vx}_{k}}-\vx_k\| = \alpha \|{ \hat{\vx}_{k}}-\vx_k\|.  \label{eq:inexact_proximal_point} 
    \end{align}
Thus, we can bound the first term in \eqref{eq:inexact_pp_decomp} by 
\begin{align}
  \langle \hat{\vx}_k-\vx_k+\eta_k\nabla f(\hat{\vx}_k), \hat{\vx}_k-\vx \rangle &\leq \|\hat{\vx}_k-\vx_k+\eta_k\nabla f(\hat{\vx}_k)\|\|\hat{\vx}_k-\vx\| \nonumber\\
  &\leq \alpha \|\hat{\vx}_k-\vx_k\|\|\hat{\vx}_k-\vx\| \nonumber\\
  &\leq \frac{\alpha}{2}\|\hat{\vx}_k-\vx_k\|^2+\frac{\alpha}{2}\|\hat{\vx}_k-\vx\|^2,\label{eq:triangle_ipp}
\end{align}
where the first inequality is due to Cauchy-Schwarz inequality, the second inequality is due to \eqref{eq:inexact_proximal_point}, and the last inequality is due to Young's inequality. 
Moreover, for the second term in \eqref{eq:inexact_pp_decomp}, we use the three-point equality to get 
\begin{equation}\label{eq:three_point}
  \langle \vx_k-\hat{\vx}_k, \hat{\vx}_k-\vx\rangle = \frac{1}{2}\|\vx_k-\vx\|^2-\frac{1}{2}\|\vx_k-\hat{\vx}_k\|^2-\frac{1}{2}\|\hat{\vx}_k-\vx\|^2.
\end{equation}
By combining \eqref{eq:inexact_pp_decomp}, \eqref{eq:triangle_ipp} and \eqref{eq:three_point}, we obtain that  
\begin{equation}\label{eq:intermediate_1}
  \eta_k \langle \nabla f(\hat{\vx}_k), \hat{\vx}_k-\vx \rangle \leq \frac{1}{2}\|\vx_k-\vx\|^2-\frac{1-\alpha}{2}\|\vx_k-\hat{\vx}_k\|^2-\frac{1-\alpha}{2}\|\hat{\vx}_k-\vx\|^2.
\end{equation}
Furthermore, by the update rule in \eqref{eq:extragradient_mixing}, we can write $\eta_k\nabla f(\hat{\vx}_k) = \vx_k-\vx_{k+1}+2\eta_k\mu (\hat{\vx}_k-\vx_{k+1})$. This implies that, for any $\vx\in \reals^d$, 
\begin{align}
  &\phantom{{}={}}\eta_k \langle \nabla f(\hat{\vx}_k), \vx_{k+1}-\vx \rangle \nonumber\\
  &= \langle \vx_k-\vx_{k+1}, \vx_{k+1}-\vx \rangle+2\eta_k\mu \langle \hat{\vx}_k-\vx_{k+1}, \vx_{k+1}-\vx\rangle \nonumber\\
  & = \frac{\|\vx_{k}-\vx\|^2}{2}-\frac{\|\vx_k-\vx_{k+1}\|^2}{2}-\frac{1+2\eta_k\mu}{2}\|\vx_{k+1}-\vx\|^2 + {\eta_k\mu}\|\hat{\vx}_{k}-\vx\|^2-{\eta_k\mu}\|\hat{\vx}_k-\vx_{k+1}\|^2, \label{eq:intermediate_2}
\end{align}
where the last equality comes from the three-point equality.  
Thus, by combining \eqref{eq:intermediate_1} with $\vx = \vx_{k+1}$ and \eqref{eq:intermediate_2} with $\vx=\vx^*$, we get
\begin{equation}\label{eq:linearized_loss}
  \begin{aligned}
    \eta_k \langle \nabla f(\hat{\vx}_k),\hat{\vx}_k - \vx^* \rangle 
    &= \eta_k \langle \nabla f(\hat{\vx}_k),{\vx}_{k+1} - \vx^* \rangle 
    +\eta_k \langle \nabla f(\hat{\vx}_k), \hat{\vx}_{k} - \vx_{k+1} \rangle \\
    &\leq \frac{1}{2}\|\vx_{k}-\vx^*\|^2-\bcancel{\frac{1}{2}\|\vx_k-\vx_{k+1}\|^2}-\frac{1+2\eta_k\mu}{2}\|\vx_{k+1}-\vx^*\|^2 \\
    &\phantom{{}={}} + {\eta_k\mu}\|\hat{\vx}_{k}-\vx^*\|^2-{\eta_k\mu}\|\hat{\vx}_k-\vx_{k+1}\|^2 \\
    &\phantom{{}={}} + \bcancel{\frac{1}{2}\|\vx_k-\vx_{k+1}\|^2}-\frac{1-\alpha}{2}\|\vx_k-\hat{\vx}_k\|^2-\frac{1-\alpha}{2}\|\hat{\vx}_k-\vx_{k+1}\|^2. %
  \end{aligned}
\end{equation}
Since $f$ is $\mu$-strongly convex, we have 
\begin{equation}\label{eq:strong_monotone}
\langle \nabla f(\hat{\vx}_k),\hat{\vx}_k - \vx^* \rangle = \langle \nabla f(\hat{\vx}_k)-\nabla f({\vx^*}),\hat{\vx}_k - \vx^* \rangle \geq \mu\|\hat{\vx}_k - \vx^*\|^2.
\end{equation}
Combining \eqref{eq:linearized_loss} and \eqref{eq:strong_monotone} and rearranging the terms, we obtain 
\begin{align}
  \frac{1+2\eta_k\mu}{2}\|\vx_{k+1}-\vx^*\|^2 &\leq \frac{1}{2}\|\vx_{k}-\vx^*\|^2-\frac{1-\alpha}{2}\|\vx_k-\hat{\vx}_k\|^2-\Bigl(\frac{1-\alpha}{2}+{\eta_k\mu}\Bigr)\|\hat{\vx}_k-\vx_{k+1}\|^2 \nonumber\\
  &\leq \frac{1}{2}\|\vx_{k}-\vx^*\|^2-\frac{1-\alpha}{2}\|\vx_k-\hat{\vx}_k\|^2. \label{eq:intermediate_3}
\end{align}
Since $\alpha<1$, the last term in \eqref{eq:intermediate_3} is negative and \eqref{eq:IEQN_contraction} follows immediately. Moreover, since $\frac{1}{2}\|\vx_{k+1}-\vx^*\|^2\leq \frac{1+2\eta_k\mu}{2}\|\vx_{k+1}-\vx^*\|^2$, we obtain from \eqref{eq:intermediate_3} that 
\begin{equation}\label{eq:intermediate_4}
  \frac{1-\alpha}{2}\|\vx_k-\hat{\vx}_k\|^2 \leq \frac{1}{2}\|\vx_{k}-\vx^*\|^2-\frac{1}{2}\|\vx_{k+1}-\vx^*\|^2.
\end{equation}
by summing~\eqref{eq:intermediate_4} over $k=0,1,\dots,N-1$, we can get 
\begin{equation*}
  \sum_{k=0}^{N-1} \frac{1-\alpha}{2}\|\hat{\vx}_{k}-\vx_k\|^2 \leq \frac{1}{2}\|\vx_0-\vx^*\|^2 - \frac{1}{2}\|\vx_N-\vx^*\|^2 \leq \frac{1}{2}\|\vx_0-\vx^*\|^2, 
\end{equation*}
which implies \eqref{eq:sum_of_square}. The proof is complete. 
\end{proof}

\subsection{Proof of Lemma~\ref{lem:step size_lb}}\label{appen:step size_lb}
If $k\notin \mathcal{B}$, by definition, the line search scheme accepts the initial trial step size at the $k$-th iteration, which means $\eta_k = \sigma_k$. Otherwise, if $k \in \mathcal{B}$, recall that $\tilde{\vx}_k$ is the last rejected point in the line search scheme, which is computed from \eqref{eq:linear_solver_update} using step size $\tilde{\eta}_k = \eta_k/\beta$. 
This means that the pair $(\tilde{\vx}_k,\tilde{\eta}_k)$ does not satisfy \eqref{eq:step size_condition}, i.e., $\tilde{\eta}_k\|\nabla f({\tilde{\vx}_k})-\nabla f(\vx_k)-\mB_k({\tilde{\vx}_k}-\vx_k)\| > \alpha_2 \|{\tilde{\vx}_k}-\vx_k\|$, which implies
  \begin{equation*}
    \eta_k = \beta \tilde{\eta}_k > \frac{\alpha_2 \beta\|{\tilde{\vx}_k}-\vx_k\|}{\|\nabla f({\tilde{\vx}_k})-\nabla f(\vx_k)-\mB_k({\tilde{\vx}_k}-\vx_k)\|}.
  \end{equation*}
  This proves the first inequality in \eqref{eq:step size_lower_bound}.
  To prove the second inequality, 
  recall from \eqref{eq:x_plus_update} that $\hat{\vx}_k$ and $\tilde{\vx}_k$ can be regarded as the inexact solution of the linear system of equations 
  \begin{equation*}
    (\mI+{\eta_k}\mB_k)({\vx}-\vx_k) = -\eta_k\nabla f(\vx_k) \quad \text{and} \quad (\mI+\tilde{\eta}_k\mB_k)({\vx}-\vx_k) = -\tilde{\eta}_k\nabla f(\vx_k),
  \end{equation*}
  respectively. 
  Define $\hat{\vx}_k^* = \vx_k - {\eta}_k(\mI+ {\eta}_k\mB_k)^{-1} \nabla f(\vx_k)$ and $\tilde{\vx}_k^* = \vx_k - \tilde{\eta}_k(\mI+ \tilde{\eta}_k\mB_k)^{-1} \nabla f(\vx_k)$, i.e., the exact solutions of the above linear systems. Since $(\hat{\vx}_k,\eta_k)$ and $(\tilde{\vx}_k,\tilde{\eta}_k)$ satisfy the condition in \eqref{eq:x_plus_update}, we have 
  \begin{equation}\label{eq:inexact_condition_linear}
    \|(\mI+\eta_k\mB_k)(\hat{\vx}_k-\hat{\vx}_k^*)\| \leq \alpha_1 \|\hat{\vx}_k - \vx_k\| \quad \text{and} \quad  \|(\mI+\tilde{\eta}_k\mB_k)(\tilde{\vx}_k-\tilde{\vx}_k^*)\| \leq \alpha_1 \|\tilde{\vx}_k - \vx_k\|.
  \end{equation} 
  We divide the proof of the second inequality {in \eqref{eq:step size_lower_bound}} into the following three steps.
    First, we  show that 
      \begin{align}
                  (1-\alpha_1)\|\hat{\vx}_k - \vx_k\| &\leq \|\hat{\vx}_k^*-{\vx}_k\| \leq (1+\alpha_1)\|\hat{\vx}_k - \vx_k\|, \label{eq:relation_exact_inexact1}\\ (1-\alpha_1)\|\tilde{\vx}_k - \vx_k\| &\leq \|\tilde{\vx}_k^*-{\vx}_k\| \leq (1+\alpha_1)\|\tilde{\vx}_k - \vx_k\|.\label{eq:relation_exact_inexact2}
      \end{align}
    In the following, we will only prove \eqref{eq:relation_exact_inexact1}, since the proof of \eqref{eq:relation_exact_inexact2} follows similarly.
    Using the fact that $\mB_k \in \semiS^d$, we have
    $
      \|(\mI+\eta_k\mB_k)(\hat{\vx}_k-\hat{\vx}_k^*)\| \geq \|\hat{\vx}_k-\hat{\vx}_k^*\|.
    $
    Hence, combining this with \eqref{eq:inexact_condition_linear}, we get $\|\hat{\vx}_k-\hat{\vx}_k^*\| \leq \alpha_1 \|\hat{\vx}_k - \vx_k\|$. It then follows from the triangle inequality that
    \begin{align*}
      \|\hat{\vx}_k^*-\vx_k\| \leq \|\hat{\vx}_k - \vx_k\| + \|\hat{\vx}_k^*-\hat{\vx}_k\| \leq (1+\alpha_1)\|\hat{\vx}_k - \vx_k\|, \\
      \|\hat{\vx}_k^*-\vx_k\| \geq \|\hat{\vx}_k - \vx_k\| -  \|\hat{\vx}_k^*-\hat{\vx}_k\| \geq (1-\alpha_1)\|\hat{\vx}_k - \vx_k\|,
    \end{align*}
    which proves \eqref{eq:relation_exact_inexact1}.
Next, we show that 
   \begin{equation}\label{eq:relation_displacement}
    \|{\tilde{\vx}^*_k}-\vx_k\|\leq \frac{1}{\beta}\|{\hat{\vx}^*_k}-\vx_k\|.
   \end{equation}
   To see this, we can compute 
  \begin{equation*}
    \|{\tilde{\vx}^*_k}-\vx_k\| \!=\! \|\tilde{\eta}_k(\mI+ \tilde{\eta}_k\mB_k)^{-1} \nabla f(\vx_k)\| \!\leq\! \|\tilde{\eta}_k(\mI+ {\eta}_k\mB_k)^{-1} \nabla f(\vx_k)\|\! =\! \frac{\tilde{\eta}_k}{{\eta}_k}\|{\hat{\vx}^*_k}-\vx_k\| = \frac{1}{\beta} \|{\hat{\vx}^*_k}-\vx_k\|,
  \end{equation*}
  where we used the fact that $\mI+ \tilde{\eta}_k\mB_k \succeq \mI+ {\eta}_k\mB_k$ in the first inequality. 
  Finally, by combining \eqref{eq:relation_exact_inexact1}, \eqref{eq:relation_exact_inexact2}, and \eqref{eq:relation_displacement}, we obtain 
  \begin{equation*}
    \|{\tilde{\vx}_k}-\vx_k\| \leq \frac{1}{1-\alpha_1}\|\tilde{\vx}_k^*-{\vx}_k\| \leq \frac{1}{\beta(1-\alpha_1)}\|\hat{\vx}_k^*-{\vx}_k\| \leq \frac{1+\alpha_1}{\beta(1-\alpha_1)}\|\hat{\vx}_k^*-{\vx}_k\|. 
  \end{equation*}
  This completes the proof.

\subsection{Proof of Lemma~\ref{lem:stepsize_bnd}}\label{appen:stepsize_bnd}
Recall that in Lemma~\ref{lem:step size_lb}, we proved that $\eta_k = \sigma_k$ if $k\notin \mathcal{B}$ and $\eta_k> \frac{\alpha_2 \beta \|\vs_k\|}{\|\vy_k-\mB_k\vs_k\|}$ otherwise, where $\vy_k \triangleq \nabla f(\tilde{\vx}_k) -\nabla f({\vx_k})$ and $\vs_k \triangleq \tilde{\vx}_k-\vx_k$. 
Using the observations above, we can write 
\begin{equation}\label{eq:intermediate_bound}
  \begin{aligned}
    \sum_{k=0}^{N-1}\frac{1}{\eta_k^2} 
  = \sum_{k\notin \mathcal{B}}\frac{1}{\eta_k^2} + \sum_{k\in \mathcal{B}}\frac{1}{\eta_k^2} &\leq
  \sum_{k\notin \mathcal{B}}\frac{1}{\sigma_k^2} + \frac{1}{\alpha_2^2\beta^2}\sum_{k\in \mathcal{B}}  \frac{\|\vy_k-\mB_k\vs_k\|^2}{\|\vs_k\|^2} \\
  &= \frac{1}{\sigma_0^2} + 
  \beta^2\sum_{k\notin \mathcal{B},k\geq 1}\frac{1}{\eta_{k-1}^2}+\frac{1}{\alpha_2^2\beta^2}\sum_{k\in \mathcal{B}}    \frac{\|\vy_k-\mB_k\vs_k\|^2}{\|\vs_k\|^2},
  \end{aligned}
\end{equation}
where we used $\sigma_k = \eta_{k-1}/\beta$ for $k\geq 1$ in the last equality. 
Since we have 
\begin{equation*}
  \sum_{k\notin \mathcal{B},k\geq 1}\frac{1}{\eta_{k-1}^2} \leq \sum_{k= 1}^{N-1}\frac{1}{\eta_{k-1}^2} \leq \sum_{k= 0}^{N-1}\frac{1}{\eta_{k}^2},
\end{equation*}
by rearranging and simplifying the terms in \eqref{eq:intermediate_bound}, we arrive at the inequality in \eqref{eq:goal}.

\section{Proof of Theorem~\ref{thm:main}}
In this section, we formally prove Lemmas~\ref{lem:stepsize_const_bound}-\ref{lem:comparator} used in Theorem~\ref{thm:main}. As discussed in the main text, throughout the proof, we assume that every call of $\mathsf{ExtEvec}$ is successful, which happens with probability at least $1-p$. Specifically, 
since the $\mathsf{ExtEvec}$ oracle has a failure probability of $q_t = \nicefrac{p}{2.5(t+1)\log^2(t+1)}$ in the $t$-th round, we can use the union bound to upper bound the total failure probability by 
\begin{equation*}
  \sum_{t=1}^{T-1} q_t = \frac{p}{2.5}\sum_{t=2}^{T} \frac{1}{t \log^2 t} \leq  \frac{p}{2.5}\sum_{t=2}^{\infty} \frac{1}{t \log^2 t} \leq \frac{p}{2.5}\left(\frac{1}{2\log^2 2}+\int_{2}^{+\infty}\frac{1}{t\log^2t}\,dt\right) \leq p.
\end{equation*}
As a result, we always have $\mB_t \in \mathcal{Z}'$, i.e., the eigenvalue of $\mB_t$ is bounded between $\frac{\mu}{2}$ and $L_1+\frac{\mu}{2}$. This property will be used in the proof of Lemma~\ref{lem:stepsize_const_bound} and Lemma~\ref{lem:small_loss}. 

\subsection{Proof of Lemma~\ref{lem:stepsize_const_bound}}
We present the general version of Lemma~\ref{lem:stepsize_const_bound} below that applies for any $\alpha_2\in (0,1)$ and $\beta\in(0,1)$.
\begingroup
\def\thetheorem{\ref{lem:stepsize_const_bound}}
\begin{lemma}
  For any $k\geq 0$, we have $\eta_k \geq \alpha_2\beta/L_1$.
\end{lemma}
\addtocounter{theorem}{-1}
\endgroup 
\begin{proof}
  We first establish that $\eta_k \geq \alpha_2\beta/L_1$ for $k\in \mathcal{B}$. To see this, suppose $k\in \mathcal{B}$ and recall from Lemma~\ref{lem:step size_lb} that
\begin{equation}\label{eq:step size_lb_appen}
  \eta_k > \frac{\alpha_2 \beta\|{\tilde{\vx}_k}-\vx_k\|}{\|\nabla f({\tilde{\vx}_k})-\nabla f(\vx_k)-\mB_k({\tilde{\vx}_k}-\vx_k)\|}.
\end{equation}
By the fundamental theorem of calculus, we can write $\nabla f({\tilde{\vx}_k})-\nabla f(\vx_k) = \bar{\mH}_k (\tilde{\vx}_k-\vx_k)$, where $\bar{\mH}_k = \int_{0}^1 \nabla^2 f(t\tilde{\vx}_k+(1-t)\vx_k) \,dt$.  
  Since we have $\mu\mI \preceq \nabla^2 f (\vx) \preceq L_1\mI$ for all $\vx\in \reals^d$ by Assumption~\ref{assum:smooth_SC}, we get $\mu\mI \preceq \bar{\mH}_k \preceq L_1 \mI$. Moreover, since $ \frac{\mu}{2}\mI \preceq {\mB}_k \preceq (L_1+\frac{\mu}{2}) \mI$, we further have $(-L_1+\frac{\mu}{2})\mI \preceq \bar{\mH}_k - \mB_k \preceq (L_1-\frac{\mu}{2})\mI$, which implies $\|\bar{\mH}_k - \mB_k\|_{\op} \le L_1-\frac{\mu}{2} \leq L_1$. Thus, we have   
  \begin{equation*}
    \|\nabla f({\tilde{\vx}_k})-\nabla f(\vx_k)-\mB_k({\tilde{\vx}_k}-\vx_k)\| = \|(\bar{\mH}_k - \mB_k)(\tilde{\vx}_k-\vx_k)\| \leq L_1\|\tilde{\vx}_k-\vx_k\|,
  \end{equation*}
  which proves that $\eta_k > {\alpha_2 \beta}/{L_1}$ from \eqref{eq:step size_lb_appen}. 

  Now we can prove that $\eta_k \geq \alpha_2\beta/L_1$ for all $k\geq 0$ by induction. To show that this holds true for $k=0$, we distinguish two cases. If $0\notin \mathcal{B}$, then we have $\eta_0 = \sigma_0 > {\alpha_2 \beta}/{L_1}$ by our choice of $\sigma_0$. Otherwise, if $0\in \mathcal{B}$, then it directly follows from our result in the previous paragraph. Moreover, assume that $\eta_{l-1} \geq \alpha_2\beta/L_1$ where $l\geq 1$. Similarly, we again distinguish two cases: if $l\notin \mathcal{B}$, then we have $\eta_l = \sigma_l = \eta_{l-1}/\beta >\alpha_2/L_1>\alpha_2\beta/L_1$; otherwise, if $l\in \mathcal{B}$, it follows from the result above that $\eta_l \geq \alpha_2\beta/L_1$. This completes the induction. 
\end{proof}

\subsection{Proof of Lemma~\ref{lem:small_loss}}
\label{appen:small_loss}
Recall that our Hessian approximation update in Subroutine~\ref{alg:hessian_approx} is a direct instantiation of the general projection-free online learning algorithm described in Section~\ref{subsubsec:projection_free_online_learning}. 
Therefore, we first present the regret analysis of the general algorithm in Lemma~\ref{lem:regret_reduction}. 
For completeness, the pseudocode of the general algorithm is also given in Algorithm~\ref{alg:projection_free_online_learning}.   
\begin{lemma}\label{lem:regret_reduction}
  Let $\{\vx_t\}_{t=0}^{T-1}$ be the iterates generated by Algorithm~\ref{alg:projection_free_online_learning}. Then we have $\vx_t \in (1+\delta)\calC$ for $t=0,1,\dots,T-1$. Also, for any $\vx\in \mathcal{C}$, we have 
  \begin{align}
      \langle \vg_t, \vx_t-\vx \rangle \leq \langle \tilde{\vg}_t, \vw_t-\vx \rangle %
      \leq \frac{1}{2\rho}\|\vw_t-\vx\|^2_2-\frac{1}{2\rho}\|\vw_{t+1}-\vx\|^2_2+ \frac{\rho}{2}\|\tilde{\vg}_t\|_2^2,\label{eq:regret_reduction} %
  \end{align}
  and  
  \begin{equation}\label{eq:surrogate_loss}
      \|\tilde{\vg}_t\| \leq \|\vg_t\|+|\langle \vg_t, \vx_t\rangle|\|\vs_t\|.
  \end{equation}
  \end{lemma}
\begin{algorithm}[!t]\small
      \caption{Projection-Free Online Learning}\label{alg:projection_free_online_learning}
      \begin{algorithmic}[1]
          \STATE \textbf{Input:} Initial point $\vw_0 \in \mathcal{B}_R(0)$, step size $\rho>0$, $\delta>0$
          \FOR{$t=0,1,\dots T-1$}
          \STATE Query the oracle $(\gamma_t,\vs_t) \leftarrow \mathsf{SEP}(\vw_t;\delta)$
          \IF[Case I: we have $\vw_t\in (1+\delta)\mathcal{C}$]{$\gamma_t \leq 1$}
            \STATE Set $\vx_t \leftarrow \vw_t$ and play the action $\vx_t$ 
            \STATE Receive the loss $\ell_t(\vx_t)$ and the gradient $\vg_t = \nabla \ell_t(\vx_t)$
            \STATE Set $\tilde{\vg}_t \leftarrow \vg_t$
          \ELSE[Case II: we have $\vw_t/\gamma_t\in (1+\delta)\mathcal{C}$] 
            \STATE Set $\vx_t \leftarrow \vw_t/\gamma_t$ and play the action $\vx_t$ 
          \STATE Receive the loss $\ell_t(\vx_t)$ and the gradient $\vg_t = \nabla \ell_t(\vx_t)$
          \STATE Set $\tilde{\vg}_t \leftarrow \vg_t+\max\{0, -\langle \vg_t, \vx_t \rangle\} \vs_t$
          \ENDIF
          \STATE Update $\vw_{t+1} \leftarrow \frac{R}{\max\{\|\vw_t-\rho \tilde{\vg}_t\|_2,R\}} (\vw_t-\rho \tilde{\vg}_t)$ 
          \COMMENT{Euclidean projection onto $\mathcal{B}_{R}(0)$}
          \ENDFOR
      \end{algorithmic}
    \end{algorithm}
\begin{proof}
We distinguish two cases depending on the outcome of $\mathsf{SEP}(\vw_t;\delta)$. 
\begin{itemize}
  \item If $\gamma_t \leq 1$, By Definition~\ref{def:gauge} we have $\vw_t \in (1+\delta) \mathcal{C}$. According to Algorithm~\ref{alg:projection_free_online_learning}, we have $\vx_t=\vw_t\in (1+\delta) \mathcal{C}$ and $\tilde{\vg}_t = \vg_t$, and thus the first inequality in \eqref{eq:regret_reduction} and the inequality in \eqref{eq:surrogate_loss} trivially hold. 
  \item Otherwise, if $\gamma_t>1$, By Definition~\ref{def:gauge} we have $\vw_t/\gamma_t \in (1+\delta)\mathcal{C}$ and $\langle \vs_t, \vw_t-\vx \rangle \geq {\gamma_t-1}$  $ \forall\vx\in \mathcal{C}$. According to Algorithm~\ref{alg:projection_free_online_learning}, we have $\vx_t  = \vw_t/\gamma_t \in (1+\delta)\mathcal{C}$ and $\tilde{\vg}_t = \vg_t+\max\{0, -\langle \vg_t, \vx_t \rangle\} \vs_t$. To prove the first inequality in \eqref{eq:regret_reduction}, note that for any $\vx\in \mathcal{C}$,
  \begin{align*}
    \langle \tilde{\vg}_t, \vw_t-\vx \rangle &= \langle\vg_t+\max\{0, -\langle \vg_t, \vx_t \rangle\} \vs_t, \vw_t-\vx \rangle \\
    & = \langle \vg_t, \gamma_t\vx_t-\vx \rangle + \max\{0, -\langle \vg_t, \vx_t \rangle\} \langle \vs_t, \vw_t-\vx\rangle \\
    & \geq \langle \vg_t, \vx_t-\vx\rangle + (\gamma_t-1)\langle\vg_t,\vx_t\rangle + (\gamma_t-1)\max\{0, -\langle \vg_t, \vx_t \rangle\}  \\
    & \geq \langle \vg_t, \vx_t-\vx\rangle, 
  \end{align*}
  where we used $\langle \vs_t, \vw_t-\vx\rangle \geq \gamma_t -1$ in the first inequality. 
  Also, by the triangle inequality we obtain 
  \begin{equation*}
    \|\tilde{\vg}_t\|_2 =\|\vg_t+\max\{0, -\langle \vg_t, \vx_t \rangle\} \vs_t\|_2 %
    \leq  \|\vg_t\|_2+|\langle \vg_t, \vx_t \rangle|\|\vs_t\|_2,
  \end{equation*}
  which proves \eqref{eq:surrogate_loss}. 
\end{itemize}
Finally, from the update rule of $\vw_{t+1}$, for any $\vx\in \mathcal{C} \subset \mathcal{B}_R(0)$ we have 
$
  \langle\vw_t-\rho\tilde{\vg}_t-\vw_{t+1}, \vw_{t+1}-\vx\rangle \geq 0
$, which further implies that 
\begin{align}
  \langle \tilde{\vg}_t, \vw_{t}-\vx \rangle &\leq  \langle \tilde{\vg}_t, \vw_{t}-\vw_{t+1}\rangle+ \frac{1}{\rho}\langle \vw_t-\vw_{t+1}, \vw_{t+1}-\vx \rangle \nonumber\\
  & =  \langle \tilde{\vg}_t, \vw_{t}-\vw_{t+1}\rangle+ \frac{1}{2\rho}\|\vw_t-\vx\|_2^2-\frac{1}{2\rho}\|\vw_{t+1}-\vx\|_2^2-\frac{1}{2\rho}\|\vw_t-\vw_{t+1}\|_2^2 \nonumber\\
  &\leq \frac{1}{2\rho}\|\vw_t-\vx\|_2^2-\frac{1}{2\rho}\|\vw_{t+1}-\vx\|_2^2 +\frac{\rho}{2}\|\tilde{\vg}_t\|_2^2. \label{eq:regret_wt}
\end{align}
This proves the second inequality in \eqref{eq:regret_reduction}. 
\end{proof}

Next, we present the following lemma showing a smooth property of the loss function $\ell_k$. {It is similar to the standard inequality $\frac{1}{2L_1} \|\nabla g (\vx)\|^2 \leq g(\vx) - g^* $ for a $L_1$-smooth function $g$, where $g^*$ denotes the minimum of $g$.} This will be the key to proving a constant upper bound on the cumulative loss incurred by Subroutine~\ref{alg:hessian_approx}.
\begin{lemma}\label{lem:loss}
  Recall the loss function $\ell_k$ defined in \eqref{eq:loss_of_Hessian}. For $k\in \mathcal{B}$, we have
  \begin{equation}\label{eq:gradient}
    \nabla \ell_k(\mB) = \frac{1}{2\|\vs_k\|^2}\left(-\vs_k(\vy_k-\mB\vs_k)^\mathsf{T}-(\vy_k-\mB\vs_k)\vs_k^\mathsf{T}\right).
  \end{equation}
  Moreover, for any $\mB\in \mathbb{S}^d$, it holds that 
    \begin{equation}\label{eq:gradient_norm_bound}
        \|\nabla \ell_k(\mB)\|_F \leq \|\nabla \ell_k(\mB)\|_* \leq \sqrt{2\ell_k(\mB)},
    \end{equation} 
    where $\|\cdot\|_F$ and $\|\cdot\|_{*}$ denote the Frobenius norm and the nuclear norm, respectively.  
\end{lemma}
\begin{proof}
  The expression in \eqref{eq:gradient} follows from direct calculation. The first inequality in \eqref{eq:gradient_norm_bound} follows from the fact that $\|\mA\|_F \leq \|\mA\|_*$ for any matrix $\mA\in \mathbb{S}^d$. For the second inequality, note that 
  \begin{align*}
    \|\nabla \ell_t(\mB)\|_* &  \leq \frac{1}{2\|\vs_t\|^2} \left( \|\vs_t(\vy_t-\mB\vs_t)^\mathsf{T}\|_*+\|(\vy_t-\mB\vs_t)\vs_t^\mathsf{T}\|_*\right) \\
    &\leq \frac{1}{\|\vs_t\|^2} \|\vy_t-\mB\vs_t\|\|\vs_t\| = \frac{\|\vy_t-\mB\vs_t\|}{\|\vs_t\|} = \sqrt{2\ell_t(\mB)}, 
  \end{align*}
  where in the first inequality we used the triangle inequality, and in the second inequality we used the fact that the rank-one matrix $\vu\vv^\top$ has only one nonzero singular value $\|\vu\|\|\vv\|$ . 
\end{proof}

Now we are ready to present the proof of Lemma~\ref{lem:small_loss}. 
By letting $\vx_t = \hat{\mB}_t$, $\vx = \hat{\mB} \triangleq \frac{2}{L_1-\mu}(\mH-\frac{L_1+\mu}{2}\mI)$, $\vg_t = \mG_t \triangleq \frac{2}{L_1-\mu}\nabla \ell_t(\mB_t)$, $\tilde{\vg}_t = \tilde{\mG}_t$, $\vw_t = \mW_t$ in Lemma~\ref{lem:regret_reduction}, we obtain:  
\begin{enumerate}[(i)]
  \item $\hat{\mB}_t \in (1+\delta)\mathcal{C}$, which means $\|\hat{\mB}_t\|_{\op} \leq 1+\delta\leq 2$ since $\delta\leq 1$.
  \item It holds that 
  \begin{align}
    \langle \mG_t, \hat{\mB}_t-\hat{\mB} \rangle &\leq \frac{1}{2\rho}\|\mW_t-\hat{\mB}\|_F^2-\frac{1}{2\rho}\|\mW_{t+1}-\hat{\mB}\|_F^2+\frac{\rho}{2}\|\tilde{\mG}_t\|_F^2, \label{eq:linearized_loss_matrix} \\
    \|\tilde{\mG}_t\|_F &\leq \|\mG_t\|_F + |\langle \mG_t, \hat{\mB}_t\rangle|\|\mS_t\|_F. \label{eq:surrogate_gradient_bound}
  \end{align}
\end{enumerate}
First, note that $\|\mS_t\|_F=1$ by Definition~\ref{def:extevec} and
$|\langle \mG_t, \hat{\mB}_t\rangle| \leq \|\mG_t\|_* \|\hat{\mB}_t\|_{\op} \leq 2 \|\mG_t\|_*$. %
Together with \eqref{eq:surrogate_gradient_bound}, we get 
\begin{equation}\label{eq:surrogate_bound_loss}
\|\tilde{\mG}_t\|_F \leq \|\mG_t\|_F+2\|\mG_t\|_* \leq 3\|\mG_t\|_* \leq \frac{6}{L_1-\mu}\sqrt{2\ell_t({\mB}_t)},
\end{equation}
where we used $\mG_t = \frac{2}{L_1+\mu}\nabla \ell_t(\mB_t)$ and Lemma~\ref{lem:loss} in the last inequality. 
Furthermore, since $\ell_t$ is convex, we have 
\begin{equation*}
    \ell_t({\mB}_t) - \ell_t(\mH) \leq \langle \nabla \ell_t(\mB_t), {\mB}_t-\mH \rangle = \left(\frac{L_1-\mu}{2}\right)^2 \langle \mG_t, \hat{\mB}_t-\hat{\mB} \rangle,
\end{equation*}
where we used $\mG_t = \frac{2}{L_1-\mu}\nabla \ell_t(\mB_t)$, $\hat{\mB}_t \triangleq \frac{2}{L_1-\mu}(\mB_t-\frac{L_1+\mu}{2}\mI)$, and $\hat{\mB} \triangleq \frac{2}{L_1-\mu}(\mH-\frac{L_1+\mu}{2}\mI)$. 
Therefore, by \eqref{eq:linearized_loss_matrix} and \eqref{eq:surrogate_bound_loss} we get 
\begin{align*}
\ell_t({\mB}_t) - \ell_t(\mH) &\leq \frac{(L_1-\mu)^2}{8\rho}\|\mW_t-\hat{\mB}\|_F^2-\frac{(L_1-\mu)^2}{8\rho}\|\mW_{t+1}-\hat{\mB}\|_F^2+\frac{\rho}{2} \left(\frac{L_1-\mu}{2}\right)^2\|\tilde{\mG}_t\|_F^2\\
&\leq \frac{(L_1-\mu)^2}{8\rho}\|\mW_t-\hat{\mB}\|_F^2-\frac{(L_1-\mu)^2}{8\rho}\|\mW_{t+1}-\hat{\mB}\|_F^2+9\rho\ell_t({\mB}_t). %
\end{align*}
Since $\rho = 1/18$, by rearranging and simplifying terms in the above inequality, we obtain
\begin{equation*}
\ell_t({\mB}_t) \leq 2\ell_t({\mH})+ \frac{9(L_1-\mu)^2}{2}\|\mW_t-\hat{\mB}\|_F^2-\frac{9(L_1-\mu)^2}{2}\|\mW_{t+1}-\hat{\mB}\|_F^2.
\end{equation*}
By summing the above inequality from $t=0$ to $T-1$, we further have 
\begin{equation*}
  \sum_{t=0}^{T-1} \ell_t({\mB}_t)  \leq \frac{9(L_1-\mu)^2}{2}\|\mW_0-\hat{\mB}\|_F^2+ 2 \sum_{t=0}^{T-1} \ell_t({\mH}) = 18\|\mB_0-{\mH}\|_F^2+ 2 \sum_{t=0}^{T-1} \ell_t({\mH}),
\end{equation*}
where the last equality is due to ${\mW}_0 \triangleq \frac{2}{L_1-\mu}(\mB_0-\frac{L_1+\mu}{2}\mI)$ and $\hat{\mB} \triangleq \frac{2}{L_1-\mu}(\mH-\frac{L_1+\mu}{2}\mI)$. This completes the proof. 

\subsection{Proof of Lemma~\ref{lem:comparator}}\label{appen:comparator}
We present the general version of Lemma~\ref{lem:comparator} below that applies for any $\alpha_1,\alpha_2\in (0,1)$ with $\alpha_1+\alpha_2 <1$ and $\beta\in(0,1)$.
\begingroup
\def\thetheorem{\ref{lem:comparator}}
\begin{lemma}
We have 
\begin{equation*}
  \sum_{t=0}^{T-1} \ell_t(\mH^*) \leq \left(\frac{(1+\alpha_1)^2}{4(1-\alpha_1)^2\beta^2 (1-\alpha_1-\alpha_2)}+1+\frac{L_1}{2\alpha_2 \beta \mu}\right)L_2^2\|\vx_0-\vx^*\|^2.
\end{equation*} 
\end{lemma}
\addtocounter{theorem}{-1}
\endgroup 
\begin{proof}
By the fundamental theorem of calculus, we can write $\vy_t = \nabla f({\tilde{\vx}_t})-\nabla f(\vx_t) = \bar{\mH}_t (\tilde{\vx}_t-\vx_t)$, where $\bar{\mH}_t = \int_{0}^1 \nabla^2 f(\vx_t+ \lambda \vs_t) \,d\lambda$. Moreover, we have 
\begin{align*}
  \|\bar{\mH}_t-\mH^*\|_{\op} \leq \int_{0}^1 \|(\nabla^2 f(\vx_t+ \lambda \vs_t)-\nabla^2 f(\vx^*))\|_{\op}\, d\lambda &\leq L_2\int_{0}^1 \|\vx_t-\lambda \vs_t + \vx^*\|\, d\lambda \\
  &\leq L_2\int_{0}^1 (\|\vx_t-\vx^*\| + \lambda \|\vs_t\|)\, d\lambda \\
  & = L_2\Bigl(\|\vx_t-\vx^*\|+\frac{1}{2}\|\vs_t\|\Bigr),
\end{align*}
where we used Assumption~\ref{assum:Hessian_Lips} in the second inequality. 
Therefore, we have $\|\vy_t-\mH^*\vs_t\| = \|(\bar{\mH}_t-\mH^*)\vs_t\| \leq \|\bar{\mH}_t-\mH^*\|_{\op}\|\vs_t\| \leq L_2\|\vs_t\|\Bigl(\|\vx_t-\vx^*\|+\frac{1}{2}\|\vs_t\|\Bigr)$. This further implies that 
\begin{align}
  \sum_{t=0}^{T-1} \ell_t(\mH^*) = \sum_{t=0}^{T-1} \frac{\|\vy_t-\mH^*\vs_t\|^2}{2\|\vs_t\|^2} &\leq \frac{L_2^2}{2}\sum_{t=0}^{T-1}\Bigl(\|\vx_t-\vx^*\|+\frac{1}{2}\|\vs_t\|\Bigr)^2 \nonumber\\ 
  &\leq \frac{L_2^2}{4} \sum_{t=0}^{T-1} \|\vs_t\|^2 + L_2^2\sum_{t=0}^{T-1} \|{\vx}_t-\vx^*\|^2. \label{eq:comparator_1}
\end{align}
To bound the sum $\sum_{t=0}^{T-1} \|\vs_t\|^2$, we use Lemma~\ref{lem:step size_lb} and the inequality in \eqref{eq:sum_of_square} to get 
\begin{equation}\label{eq:comparator_2}
  \sum_{t=0}^{T-1} \|\vs_t\|^2 = \sum_{t=0}^{T-1} \|\tilde{\vx}_t-\vx_t\|^2 \leq \frac{(1+\alpha_1)^2}{\beta^2 (1-\alpha_1)^2} \sum_{t=0}^{T-1} \|\hat{\vx}_t-\vx_t\|^2 \leq \frac{(1+\alpha_1)^2\|\vx_0-\vx^*\|^2}{(1-\alpha_1)^2\beta^2 (1-\alpha_1-\alpha_2)}. 
\end{equation}
To bound the sum $\sum_{t=0}^{T-1}\|{\vx}_t-\vx^*\|^2$, we use the linear convergence result in Part (a) of Theorem~\ref{thm:main}:
\begin{equation}\label{eq:comparator_3}
  \sum_{t=0}^{T-1} \|\vx_t-\vx^*\|^2 \leq \|\vx_0-\vx^*\|^2\sum_{t=0}^{T-1} \left(1+\frac{\alpha_2\beta\mu}{L_1}\right)^{-t} \leq  \|\vx_0-\vx^*\|^2\left(1+\frac{L_1}{2\alpha_2 \beta \mu}\right).
\end{equation}
Lemma~\ref{lem:comparator} follows immediately from \eqref{eq:comparator_1}, \eqref{eq:comparator_2}, and \eqref{eq:comparator_3}. 
\end{proof}

\section{Characterizing the Computational Cost}\label{appen:computational_cost}

In this section, we characterize the computational cost of our QNPE method. 

\subsection{Implementation of \texorpdfstring{$\mathsf{LinearSolver}$}{LinearSolver}  Oracle}\label{appen:CR}

In this section, we describe an efficient implementation of the $\mathsf{LinearSolver}$ oracle in Definition~\ref{def:linear_solver}. On a high level, we run the conjugate residual (CR) method \citep{saad2003iterative} to solve the linear system $\mA\vs = \vb$ with $\vs_0 = 0$, and returns the iterate $\vs_k$ once it satisfies $\|\mA\vs_k-\vb\|\leq \alpha \|\vs_k\|$. CR is a Krylov subspace method similar to the better known conjugate gradient (CG) method. 
In particular, it is designed to minimize the norm of the residual vector $\vr_k := \vb-\mA\vs_k$ over the Krylov subspace and thus is more suitable for our purpose.
For completeness, the full algorithm is shown in Subroutine~\ref{alg:CRM}. Note that in Line~\ref{line:Ap} we can compute $\mA\vp_{k+1}$ from $\mA\vr_{k+1}$ and $\mA\vp_k$ without an additional matrix-vector product, and hence $\mathsf{LinearSolver}$ requires exactly two matrix-vector products in each iteration. 

\begin{subroutine}[!t]\small
  \caption{$\mathsf{LinearSolver}(\mA,\vb; \alpha)$}\label{alg:CRM}
  \begin{algorithmic}[1]
      \STATE \textbf{Input:} $\mA \in \semiS^d$, $\vb\in \reals^d$, $0<\alpha<1$
      \STATE \textbf{Initialize:} $\vs_0 \leftarrow 0$, $\vr_0 \leftarrow \vb-\mA\vs_0$, $\vp_0 \leftarrow \vr_0$
      \FOR{$k=0,1,\dots$}
      \IF{$\|\vr_k\|_2\leq {\alpha}\|\vs_k\|_2$}
        \STATE \textbf{Return} $\vs_k$
      \ENDIF
      \STATE $\alpha_k \leftarrow \langle \vr_k, \mA\vr_k \rangle/\langle \mA\vp_k, \mA \vp_k\rangle$ \tikzmark{top}
      \STATE $\vs_{k+1} \leftarrow \vs_k+\alpha_k\vp_k$
      \STATE $\vr_{k+1}\leftarrow \vr_k-\alpha_k\mA\vp_k$
      \STATE Compute and store $\mA\vr_{k+1}$
      \STATE $\beta_k \leftarrow \langle \vr_{k+1}, \mA\vr_{k+1} \rangle/\langle \vr_k, \mA\vr_k \rangle$
      \STATE $\vp_{k+1} \leftarrow \vr_{k+1}+\beta_k\vp_k$ 
      \STATE Compute and store $\mA\vp_{k+1} \leftarrow \mA\vr_{k+1}+\beta_k\mA\vp_k$ \label{line:Ap} \tikzmark{bottom} \tikzmark{right}
      \ENDFOR
  \end{algorithmic}
  \AddNote{top}{bottom}{right}{\color{comment}\textit{ Conjugate residual method\\\; for solving $\mA \vs = \vb$}}
\end{subroutine}

Before presenting the complexity bound of Subroutine~\ref{alg:CRM}, we first review some properties of the CR method. In the following, we let $\lambda_{\max}(\mA)$ and $\lambda_{\min}(\mA)$ denote the maximum and minimum eigenvalues of $\mA$, respectively.
\begin{proposition}\label{prop:CR}
  Let $\{\vs_k\}_{k\geq 0}$ and $\{\vr_k\}_{k\geq 0}$ be generated by Subroutine~\ref{alg:CRM}. Then the following holds:
  \begin{enumerate}[(a)]
    \item We have 
    \begin{equation*}
      \|\vr_k\| \leq 2\left(\frac{\sqrt{\kappa(\mA)}-1}{\sqrt{\kappa(\mA)}+1}\right)^k\|\vr_0\|, %
    \end{equation*}
    where $\kappa(\mA) = \lambda_{\mathrm{max}}(\mA)/\lambda_{\mathrm{min}}(\mA)$ denotes the condition number of $\mA$. 
    \item We have $\|\vs_k\| > \|\vs_{k-1}\|$ for all $k\geq 1$. 
  \end{enumerate}
\end{proposition}
\begin{proof}
  See \cite[{Section 3.1}]{greenbaum1997iterative} for the proof of Part (a) and \cite[{Theorem 2.1.6}]{Fong2011} for the proof of Part (b). 
\end{proof}
As a corollary of Proposition~\ref{prop:CR}, we obtain an sufficient condition for $\|\vr_k\|\leq {\alpha}\|\vs_k\|$. 
\begin{lemma}\label{lem:CR_ratio}
  If $\|\vr_k\|\leq \alpha \|\vr_0\|/\lambda_{\mathrm{max}}(\mA)$, then we have 
  $
    \|\vr_k\|\leq {\alpha}\|\vs_k\|.
  $
\end{lemma}
\begin{proof}
  From the update rule of Subroutine~\ref{alg:CRM}, we can compute that $\vs_1 = \frac{\vb^\top \mA \vb}{\|\mA\vb\|_2^2} \vb$, which implies
  \begin{equation*}
    \|\vs_1\| = \|\vb\|\cdot \frac{\|\mA^{1/2}\vb\|^2}{(\mA^{1/2}\vb)^\top \mA (\mA^{1/2}\vb)} \geq \frac{\|\vb\|}{\lambda_{\max}(\mA)} = \frac{\|\vr_0\|}{\lambda_{\max}(\mA)}.
  \end{equation*}
  Since $\|\vs_k\|$ is strictly increasing (cf. Proposition~\ref{prop:CR}(b)), we have $\|\vs_k\| \geq \|\vs_1\| = \frac{\|\vr_0\|}{\lambda_{\max}(\mA)}$ for any $k \geq 1$. Thus, we obtain that $\|\vr_k\|_2\leq \alpha \|\vr_0\|_2/\lambda_{\mathrm{max}}(\mA) \leq \alpha \|\vs_k\|_2$, which completes the proof. 
\end{proof}
In the following lemma, we upper bound the total number of matrix-product evaluations during one execution of Subroutine~\ref{alg:CRM}. 
\begin{lemma}\label{lem:conjugate_residual}
When Subroutine~\ref{alg:CRM} returns, the total number of matrix-vector product evaluations can be bounded by  
 $2\sqrt{\frac{\lambda_{\max}(\mA)}{\lambda_{\min}(\mA)}} \log \left(\frac{2\lambda_{\max}(\mA)}{\alpha}\right)$.
\end{lemma}

\begin{proof}
  Combining Proposition~\ref{prop:CR} and Lemma~\ref{lem:CR_ratio}, we obtain that $\|\vr_k\|_2 \leq \alpha \|\vs_k\|_2$ if 
  \begin{equation*}
    2\left(\frac{\sqrt{\kappa(\mA)}-1}{\sqrt{\kappa(\mA)}+1}\right)^k \leq \frac{\alpha}{\lambda_{\mathrm{max}}(\mA)} \quad\Leftrightarrow\quad 
    k \geq \frac{\log \left(\frac{2\lambda_{\max}(\mA)}{\alpha}\right)}{\log\left(\frac{\sqrt{\kappa(\mA)}+1}{\sqrt{\kappa(\mA)}-1}\right)}.
  \end{equation*}
  Since $\log(x) \geq (x-1)/x$ for all $x >0$, we further have  $\log\left(\frac{\sqrt{\kappa(\mA)}+1}{\sqrt{\kappa(\mA)}-1}\right) \geq \frac{2}{\sqrt{\kappa(\mA)}+1} \geq \frac{1}{\sqrt{\kappa(\mA)}}$. This completes the proof.
\end{proof}

\subsection{Implementation of \texorpdfstring{$\mathsf{ExtEvec}$}{ExtEvec} Oracle}\label{appen:SEP}
\begin{subroutine}[!t]\small
  \caption{$\mathsf{ExtEvec}(\mW;\delta,q)$}\label{alg:lanczos}
  \begin{algorithmic}[1]
      \STATE \textbf{Input:} $\mW \in \mathbb{S}^d$, $\delta>0$, $q\in (0,1)$
      \STATE \textbf{Initialize:} sample $\vv_1\in \reals^d$ uniformly from the unit sphere, $\beta_1 \leftarrow 0$, $\vv_0\leftarrow 0$
      \STATE Set $\epsilon \leftarrow \frac{\delta}{2(1+\delta)}$ and the number of iterations $N \leftarrow \min\Bigl\{\Bigl\lceil \frac{1}{4}\epsilon^{-1/2}\log\frac{11d}{q^2}+\frac{1}{2}\Bigr\rceil, d\Bigr\}$
      \FOR{$k=1,\dots,N$ \tikzmark{top}}
      \STATE Set $\vw_k \leftarrow \mW \vv_k-\beta_k \vv_{k-1}$ 
      \STATE Set $\alpha_k \leftarrow \langle \vw_k,\vv_k \rangle $ and $\vw_k \leftarrow \vw_k-\alpha_k\vv_k$
      \STATE Set $\beta_{k+1} \leftarrow \|\vw_k\|$ and $\vv_{k+1}\leftarrow \vw_k/\beta_{k+1}$
      \ENDFOR
    \STATE Form a tridiagonal matrix $\mT \leftarrow \mathsf{tridiag}(\beta_{2:N},\alpha_{1:N},\beta_{2:N})$
    \STATE \hspace{-1em}\COMMENT{Use the tridiagonal structure to compute eigenvectors of $\mT$}
    \STATE Compute $(\hat{\lambda}_1,\vz^{(1)}) \leftarrow \mathsf{MaxEvec}(\mT)$ and $(\hat{\lambda}_d, \vz^{(d)}) \leftarrow \mathsf{MinEvec}(\mT)$ \tikzmark{right}
    \STATE Set $\vu^{(1)} \leftarrow \sum_{k=1}^N z^{(1)}_k\vv_k$ and $\vu^{(d)} \leftarrow \sum_{k=1}^N z^{(d)}_k\vv_k$ \tikzmark{bottom}
    \STATE Set $\gamma \leftarrow \max\{\hat{\lambda}_1,-\hat{\lambda}_d\}$ 
    \IF{$\gamma \leq 1$}
    \STATE Return $\gamma$ and $\mS = 0$ \COMMENT{Case I: $\gamma\leq 1$, which implies $\|\mW\|_{\op} \leq 1+\delta$}
    \ELSIF{$\hat{\lambda}_1 \geq -\hat{\lambda}_d$}
      \STATE Return $\gamma$ and $\mS = \vu^{(1)}(\vu^{(1)})^\top$ \hspace{.75em}\COMMENT{Case II: $\gamma> 1$ and $\mS$ defines a separating hyperplane}
    \ELSE
      \STATE Return $\gamma$ and $\mS = -\vu^{(d)}(\vu^{(d)})^\top$ \COMMENT{Case II: $\gamma> 1$ and $\mS$ defines a separating hyperplane}
    \ENDIF
  \end{algorithmic}
  \AddNote{top}{bottom}{right}{\color{comment}\textit{\quad Lanczos method}}
\end{subroutine}

In this section, we describe an efficient implementation of the $\mathsf{ExtEvec}$ oracle in Definition~\ref{def:extevec}. As we discussed in Section~\ref{subsec:projection_free}, it is closely related to the problem of computing the extreme eigenvalues and eigenvectors of a given matrix, and thus we build our method on the classical Lanczos method with a random start, where the initial vector is chosen randomly and uniformly from the unit sphere (see, e.g., \citep{saad2011numerical,yurtsever2021scalable}). 
On a high level, given the input matrix $\mW\in \mathbb{S}^d$,  
we first run the Lanczos method for a sufficiently large number of iterations to obtain $ \vu^{(1)}$ and $\vu^{(d)}$ as the approximation of the largest and smallest eigenvector of $\mW$, respectively. We further define $\hat{\lambda}_1 = \langle \mW\vu^{(1)},\vu^{(1)}\rangle$ and $\hat{\lambda}_d = \langle \mW\vu^{(d)},\vu^{(d)}\rangle$ as an approximation of the largest and smallest eigenvalues of $\mW$, and let $\gamma = \max\{\hat{\lambda}_1,-\hat{\lambda}_d\}$. To construct the output $(\gamma,\mS)$ satisfying the conditions in Definition~\ref{def:extevec}, we distinguish two cases depending on $\gamma$. If $\gamma \leq 1$, then we are in \textbf{Case I}, where we return $\gamma$ and $\mS = 0$. Otherwise, if $\gamma > 1$, then we are in \textbf{Case II},  where we return $\gamma$ and the rank-one matrix $\mS$ given by 
\begin{equation*}
  \mS = \begin{cases}
    \vu^{(1)}(\vu^{(1)})^\top, & \text{if }\hat{\lambda}_1 \geq -\hat{\lambda}_d; \\
    -\vu^{(d)}(\vu^{(d)})^\top, & \text{otherwise.}
  \end{cases}
\end{equation*}  
For completeness, the full algorithm is shown in Subroutine~\ref{alg:lanczos}.

As we will show in Lemma~\ref{lem:extevec_bound}, to satisfy the conditions in Definition~\ref{def:extevec}, it is sufficient to run the Lanczos method for $\mathcal{O}(\sqrt{1+1/\delta}\log(d/q^2))$ iterations.  
To prove this,   
we first recall a classical result in \cite{kuczynski1992estimating} on the convergence behavior of the Lanczos method. %
\begin{proposition}[{\cite[Theorem~4.2]{kuczynski1992estimating}}]\label{prop:lanczos}
  Consider a symmetric matrix $\mW$ and let $\lambda_1(\mW)$ and $\lambda_d(\mW)$ denote its largest and smallest eigenvalues, respectively. Then after $k$ iterations of the Lanczos method with a random start, we find unit vectors $\vu^{(1)}$ and $\vu^{(d)}$  such that 
  \begin{align*}
    \mathbb{P}(\langle\mW \vu^{(1)}, \vu^{(1)}\rangle\leq \lambda_1(\mW)-\epsilon(\lambda_1(\mW)-\lambda_d(\mW))) \leq 1.648\sqrt{d}e^{-\sqrt{\epsilon}(2k-1)},\\
    \mathbb{P}(\langle\mW \vu^{(d)}, \vu^{(d)}\rangle\geq \lambda_d(\mW)+\epsilon(\lambda_1(\mW)-\lambda_d(\mW))) \leq 1.648\sqrt{d}e^{-\sqrt{\epsilon}(2k-1)},
  \end{align*}
  As a corollary, to ensure that, with probability at least $1-q$, 
  \begin{equation*}
    \langle\mW \vu^{(1)}, \vu^{(1)}\rangle> \lambda_1(\mW)-\epsilon(\lambda_1(\mW)-\lambda_d(\mW)) 
    \; \text{and} \; 
    \langle\mW \vu^{(d)}, \vu^{(d)}\rangle< \lambda_n(\mW)+\epsilon(\lambda_1(\mW)-\lambda_d(\mW)),
  \end{equation*}
  the number of iterations can be bounded by  $\lceil\frac{1}{4}\epsilon^{-1/2}\log(11d/q^2)+\frac{1}{2} \rceil$.
\end{proposition}

\begin{lemma}\label{lem:extevec_bound}
  Let $\gamma$ and $\mS$ be the output of $\mathsf{ExtEvec}(\mW;\delta,q)$ in Subroutine~\ref{alg:lanczos} after $\Bigl\lceil \frac{1}{4}\epsilon^{-1/2}\log\frac{11d}{q^2}+\frac{1}{2}\Bigr\rceil$ iterations. Then with probability at least $1-q$, they satisfy one of the following properties: 
  \begin{itemize}
    \item Case I: $\gamma \leq 1$, then we have $\|\mW\|_{\op} \leq 1+\delta$;
    \vspace{-2mm} 
    \item Case II: $\gamma>1$, then we have $\|\mW/\gamma\|_{\op} \leq 1+\delta$, $\|\mS\|_F = 1$ and $\langle \mS,\mW-\hat{\mB}\rangle \geq \gamma -1$ for any $\hat{\mB}$ such that $\|\hat{\mB}\|_{\op} \leq 1$.
  \end{itemize}
\end{lemma}
\begin{proof}
  Note that in Subroutine~\ref{alg:lanczos}, we run the Lanczos method for $\Bigl\lceil \frac{1}{4}\epsilon^{-1/2}\log\frac{11d}{q^2}+\frac{1}{2}\Bigr\rceil$ iterations, where $\epsilon = \frac{\delta}{2(1+\delta)}$. 
  Thus, by Proposition~\ref{prop:lanczos}, with probability at least $1-q$ we have 
  \begin{align}
    \hat{\lambda}_1 \triangleq \langle\mW \vu^{(1)}, \vu^{(1)}\rangle \geq  \lambda_1(\mW)-\epsilon(\lambda_1(\mW)-\lambda_d(\mW)), \label{eq:lambda_1}\\ 
    \hat{\lambda}_d \triangleq \langle\mW \vu^{(d)}, \vu^{(d)}\rangle \leq  \lambda_d(\mW)+\epsilon(\lambda_1(\mW)-\lambda_d(\mW)). \label{eq:lambda_d}
  \end{align}
  Combining \eqref{eq:lambda_1} and \eqref{eq:lambda_d}, we get 
    \begin{equation*}
      \left(1-2\epsilon\right)(\lambda_1(\mW)-\lambda_d(\mW))\leq \hat{\lambda}_1-\hat{\lambda}_d \quad \Rightarrow \quad \lambda_1(\mW)-\lambda_d(\mW)\leq \frac{1}{1-2\epsilon}(\hat{\lambda}_1-\hat{\lambda}_d).
    \end{equation*}
  By plugging the above inequality back into \eqref{eq:lambda_1} and \eqref{eq:lambda_d}, we further have 
  \begin{align}
    \lambda_1(\mW) \leq \hat{\lambda}_1+\epsilon(\lambda_1(\mW)-\lambda_d(\mW)) \leq \hat{\lambda}_1+\frac{\epsilon}{1-2\epsilon}(\hat{\lambda}_1-\hat{\lambda}_d), \label{eq:upper_bound_lambda_1} \\
    \lambda_d(\mW) \geq \hat{\lambda}_d-\epsilon(\lambda_1(\mW)-\lambda_d(\mW)) \geq \hat{\lambda}_d-\frac{\epsilon}{1-2\epsilon}(\hat{\lambda}_1-\hat{\lambda}_d). \label{eq:lower_bound_lambda_d}
  \end{align}
  Recall that $\gamma = \max\{\hat{\lambda}_1,-\hat{\lambda}_d\}$. By \eqref{eq:upper_bound_lambda_1} and \eqref{eq:lower_bound_lambda_d}, 
  we can further bound the eigenvalues of $\mW$ by  
  \begin{equation}
    \lambda_1(\mW) \leq \gamma +\frac{\epsilon}{1-2\epsilon}\cdot 2\gamma =\frac{\gamma}{1-2\epsilon} \quad \text{and} \quad \lambda_d(\mW) \geq -\gamma -\frac{\epsilon}{1-2\epsilon}\cdot 2\gamma = -\frac{\gamma}{1-2\epsilon}.
  \end{equation}
  Hence, we can see that $\|\mW\|_\op  = \max\{\lambda_1(\mW),-\lambda_d(\mW)\}\leq \gamma/(1-2\epsilon)=(1+\delta)\gamma$. 
  Now we distinguish three cases. 
  \begin{enumerate}[(a)]
    \item If $\gamma \leq 1$, then we are in \textbf{Case I} and the $\mathsf{ExtEvec}$ oracle outputs $\gamma$ and $\mS = 0$. In this case, we indeed have $\|\mW\|_\op \leq (1+\delta)\gamma \leq 1+\delta$. 
    \item If $\gamma >1$ and $\hat{\lambda}_1 \geq -\hat{\lambda}_d$, then we are in \textbf{Case II} and the $\mathsf{ExtEvec}$ oracle returns $\gamma $ and $\mS = \vu^{(1)}(\vu^{(1)})^\top$. In this case, since $\|\mW\|_\op \leq \gamma(1+\delta)$, we have $\|\mW/\gamma\|_{\op} \leq 1+\delta$. Also, since $\vu_1$ is a unit vector, we have $\|\mS\|_F = \|\vu_1\| = 1$. Finally, for any $\hat{\mB}$ such that $\|\hat{\mB}\|_\op \leq 1$, we have 
    \begin{equation*}
      \langle \mS,\mW-\hat{\mB}\rangle = \vu_1^\top{\mW} \vu_1-\vu_1^\top \hat{\mB} \vu_1 \geq \hat{\lambda}_1-1 = \gamma-1. 
    \end{equation*} 
    \item If $\gamma >1$ and $-\hat{\lambda}_d \geq \hat{\lambda}_1$, then we are also in \textbf{Case II} and the $\mathsf{ExtEvec}$ oracle returns $\gamma$ and $\mS = -\vu^{(d)}(\vu^{(d)})^\top$. The rest follows similarly as the case above.
  \end{enumerate}
  This completes the proof.
\end{proof}

\subsection{Proof of Theorem~\ref{thm:computational_cost}}
Now that we have specified the implementation details of the $\mathsf{LinearSolver}$ and $\mathsf{ExtEvec}$ oracles, we move to the proof of Theorem~\ref{thm:computational_cost}. We divide the proof into the following three lemmas, which address the number of gradient evaluations, the number of matrix-vector products in $\mathsf{LinearSolver}$, and the number of matrix-vector products in $\mathsf{ExtEvec}$, respectively. 
{
In the following, we present the general case for any $\alpha_2\in (0,1)$.}
\begin{lemma}\label{lem:line_search_complexity}
  If we run Algorithm~\ref{alg:Full_EQN} as specified in Theorem~\ref{thm:main} for $N$ iterations, then the total number of line search steps can be bounded by $2N+\log_{1/\beta}(\sigma_0 L_1/\alpha_2)$. 
\end{lemma}
\begin{proof}
Let $l_k$ denote the number of line search steps in iteration $k$. %
We first note that $\eta_k = \sigma_k \beta^{l_k-1}$ by our line search subroutine, which implies $l_k = \log_{1/\beta}(\sigma_k/\eta_k)+1$. Thus, the total number of line search steps after $N$ iterations can be bounded by 
  \begin{align}
    \sum_{k=0}^{N-1} l_k = \sum_{k=0}^{N-1} \left(\log_{1/\beta}\frac{\sigma_k}{\eta_k}+1\right) &= N+ \log_{1/\beta}\frac{\sigma_0}{\eta_0}+\sum_{k=1}^{N-1} \log_{1/\beta}\frac{\sigma_k}{\eta_k} \nonumber\\
    &= N+ \log_{1/\beta}\frac{\sigma_0}{\eta_0} + \sum_{k=1}^{N-1} \log_{1/\beta} \frac{\eta_{k-1}}{\beta \eta_k} \label{eq:using_sigma}\\
    &= 2N-1+ \log_{1/\beta}\frac{\sigma_0}{\eta_0} + \sum_{k=1}^{N-1} \log_{1/\beta} \frac{\eta_{k-1}}{\eta_k}  \nonumber\\
    & = 2N-1+ \log_{1/\beta}\frac{\sigma_0}{\eta_{N-1}}, \label{eq:bound_on_calls}
  \end{align} 
  where we used the fact that $\sigma_k = \eta_{k-1}/\beta$ for $k\geq 1$ in \eqref{eq:using_sigma}. Since we have $\eta_{N-1} \geq \alpha_2\beta/L_1$ by Lemma~\ref{lem:stepsize_const_bound}, the lemma follows immediately from \eqref{eq:bound_on_calls}.   
\end{proof}

Note that each line search step consists of one gradient evaluation. Additionally, in each iteration of Algorithm~\ref{alg:Full_EQN}, we also need to query the gradient at $\vx_k$. Thus, as a corollary of Lemma~\ref{lem:line_search_complexity}, we conclude that the total number of gradient evaluations is bounded by $3N+\log_{1/\beta}(\sigma_0 L_1/\alpha_2)$.

\begin{lemma}
  If we run Algorithm~\ref{alg:Full_EQN} as specified in Theorem~\ref{thm:main} for $N$ iterations, then the total number of matrix-vector products in $\mathsf{ExtEvec}$ can be bounded by 
  \begin{equation*}
    N_\epsilon\Bigl( \frac{1}{4}\sqrt{\frac{2L_1}{\mu}}\log\frac{70d N_\epsilon^2\log^4(N_\epsilon)}{p^2}+\frac{3}{2}\Bigr).
  \end{equation*}
\end{lemma}
\begin{proof}
  It directly follows from Lemma~\ref{lem:extevec_bound} and our choice of parameters, where $\delta = \min\{\frac{\mu}{L_1-\mu},1\}$, $\epsilon = \frac{\delta}{2(1+\delta)} \geq \frac{\mu}{2L_1}$ and $q_t = \nicefrac{p}{2.5(t+1)\log^2(t+1)} \leq \nicefrac{p}{2.5N_{\epsilon}\log^2(N_{\epsilon})}$. 
\end{proof}

  \begin{lemma}\label{thm:matrix_vector_product}
    Let $N_\epsilon$ be the minimum number of iterations required by Algorithm~\ref{alg:Full_EQN} to achieve $\|\vx_N-\vx^*\|^2 \leq \epsilon$. Then the total number of matrix-vector products in $\mathsf{LinearSolver}$ can be bounded by 
    \begin{equation*}
      2\sqrt{\frac{3L_1}{\mu}}\log \left(\frac{2 L_1 \|\vx_0-\vx^*\|^2}{\alpha_1\beta \mu \epsilon}\right) \cdot \left(2N_{\epsilon}+\log_{1/\beta} \frac{\sigma_0 L_1}{\alpha_2} \right)+2\sqrt{\frac{3L_1}{\mu}} C_0,
    \end{equation*}
    where $C_0 \triangleq \log \left(\frac{2}{\alpha_1}\left(1+\frac{3}{2}\sigma_0 L_1\right)\right)\left(\log_{1/\beta}\Bigl(\frac{\sigma_0L_1}{\alpha_2\beta}\Bigr)+1\right)$ is an constant depending on the hyperparameters. 
  \end{lemma}
  
  \begin{proof}
    Consider the $k$-th iteration. Note that in each call of $\mathsf{LinearSolver}$ in Subroutine~\ref{alg:ls}, the input matrix $\mA$ is given by $\mA = \mI + \eta_{+} \mB_k$ with $\eta_{+} \leq \sigma_k$. 
    {Therefore, we can bound  $\frac{\lambda_{\max}(\mA)}{\lambda_{\min}(\mA)} = \frac{1+\eta_+\lambda_{\max}(\mB_k)}{1+\eta_+\lambda_{\min}(\mB_k)} \leq \frac{\lambda_{\max}(\mB_k)}{\lambda_{\min}(\mB_k)}$. Moreover, since $\frac{\mu}{2}\mI \preceq \mB_k \preceq (L_1+\frac{\mu}{2})\mI$, we have $\frac{\lambda_{\max}(\mB_k)}{\lambda_{\min}(\mB_k)} \leq \frac{2L_1+\mu}{\mu} \leq \frac{3L_1}{\mu}$.}
    Hence, by Lemma~\ref{lem:conjugate_residual}, the number of matrix-vector product evaluations in each call of $\mathsf{LinearSolver}$ can be bounded by 
    \begin{equation*}
      \mathsf{MV}_k \leq 2\sqrt{\frac{\lambda_{\max}(\mA)}{\lambda_{\min}(\mA)}} \log \left(\frac{2\lambda_{\max}(\mA)}{\alpha_1}\right) \leq 2\sqrt{\frac{3L_1}{\mu}} \log \left(\frac{2}{\alpha_1}\left(1+\sigma_k\left(L_1+\frac{\mu}{2}\right)\right)\right). 
    \end{equation*}
    Moreover, since we have $\sigma_k = \eta_{k-1}/\beta$ for $k\geq 1$, we further get 
    \begin{align*}
      \mathsf{MV}_k %
      &\leq 2\sqrt{\frac{3L_1}{\mu}} \log \left(\frac{2}{\alpha_1}\left(1+\frac{\eta_{k-1}L_1}{\beta}+\frac{\eta_{k-1}\mu}{2\beta}\right)\right) \\
       & \leq 2\sqrt{\frac{3L_1}{\mu}} \log \left(\frac{2 L_1}{\alpha_1\beta \mu}\left(1+2{\eta_{k-1}\mu}\right)\right) \\
      & \leq 2\sqrt{\frac{3L_1}{\mu}} \log \left(\frac{2 L_1}{\alpha_1\beta \mu}\right) + 2\sqrt{\frac{3L_1}{\mu}} \log(1+2{\eta_{k-1}\mu}). 
    \end{align*}
    Let $l_k$ denote the number of line search steps in iteration $k$, and then we can bound the total number of matrix-vector products by $\sum_{k=0}^{N_\epsilon-1} l_k \cdot \mathsf{MV}_k$. 
    Moreover, from the proof of Lemma~\ref{lem:line_search_complexity}, we know that $l_k = \log_{1/\beta}(\sigma_k/\eta_k)+1$.
    For $k=0$, we have 
    \begin{equation*}
      l_0 \leq \log_{1/\beta}\Bigl(\frac{\sigma_0}{\eta_0}\Bigr)+1 \leq \log_{1/\beta}\Bigl(\frac{\sigma_0L_1}{\alpha_2\beta}\Bigr)+1 \quad \text{and} \quad \mathrm{MV}_0 \leq 2\sqrt{\frac{3L_1}{\mu}} \log \left(\frac{2}{\alpha_1}\left(1+\frac{3}{2}\sigma_0 L_1\right)\right),
    \end{equation*} 
    where we used the fact that $\eta_0 >\frac{\alpha_2 \beta}{L_1}$ by Lemma~\ref{lem:stepsize_const_bound}. 
    On the other hand, %
    we first show that  
    \begin{equation}\label{eq:stepsize_upper_bound}
      \prod_{k=0}^{N_\epsilon-2} (1+2\eta_k\mu) \leq \frac{\|\vx_0-\vx^*\|^2}{\epsilon}.
    \end{equation}
    To see this, note that by Proposition~\ref{prop:IEQN}, it holds that $\|\vx_{N}-\vx^*\|^2 \leq \|\vx_0-\vx^*\|^2 \prod_{k=0}^{N-1} (1+2\eta_k\mu)^{-1}$. Then \eqref{eq:stepsize_upper_bound} follows from the fact that $N_{\epsilon}$ is the minimum number of iterations to achieve $\|\vx_N-\vx^*\|^2 \leq \epsilon$. Thus, we have 
    \begin{align*}
      \sum_{k=1}^{N-1} l_k \cdot \mathsf{MV}_k &\leq 2\sqrt{\frac{3L_1}{\mu}} \log \left(\frac{2 L_1}{\alpha_1\beta \mu}\right) \sum_{k=1}^{N-1} l_k + 2\sqrt{\frac{3L_1}{\mu}} \sum_{k=1}^{N-1} \log(1+2{\eta_{k-1}\mu}) \cdot l_k \\
      &\leq 2\sqrt{\frac{3L_1}{\mu}} \log \left(\frac{2 L_1}{\alpha_1\beta \mu}\right) \sum_{k=1}^{N-1} l_k + 2\sqrt{\frac{3L_1}{\mu}} \sum_{k=1}^{N-1} \log(1+2{\eta_{k-1}\mu}) \cdot \sum_{k=1}^{N-1} l_k \\
      &\leq 2\sqrt{\frac{3L_1}{\mu}}\log \left(\frac{2 L_1 \|\vx_0-\vx^*\|^2}{\alpha_1\beta \mu \epsilon}\right) \cdot \left(2N+\log_{1/\beta} \frac{\sigma_0 L_1}{\alpha_2} \right),
    \end{align*}
    where we used \eqref{eq:stepsize_upper_bound} and Lemma~\ref{lem:line_search_complexity} in the last inequality. 
    The proof is complete. 
  \end{proof}

\section{Additional Discussions}
\subsection{The Cost of Euclidean Projection}\label{appen:ogd_regret}
Recall that in our learning problem in Section~\ref{sec:no_regret}, the action set is given by $\mathcal{Z}' \triangleq \{\mB\in \semiS^d: \frac{\mu}{2}\mI \preceq \mB \preceq (L_1+\frac{\mu}{2}) \mI\}$. 
The Euclidean projection on this set has a closed form solution. Specifically, for a given matrix $\mA \in \mathbb{S}^d$, we first compute its eigendecomposition $\mA = \mV \mLambda \mV^\top$, where $\mV$ is an orthogonal matrix and $\mLambda = \mathrm{diag}(\lambda_1,\dots,\lambda_d)$ is a diagonal matrix. Then the Euclidean projection of $\mA$ onto $\mathcal{Z}'$ is given by $\mV \mLambda' \mV^\top$, where $\mLambda'$ is a diagonal matrix with the diagonals being $\lambda'_k = \min\{L_1+\frac{\mu}{2},\max\{\frac{\mu}{2},\lambda_k\}\}$ for $1\leq k \leq d$. 
However, note that the complexity of computing the eigendecomposition is $\bigO(d^3)$, which could be prohibitive in practice. On the contrary, our projection-free online learning algorithm relies on the $\mathsf{ExtEvec}$ oracle, which can be implemented by using matrix-vector products as we detailed in Section~\ref{appen:SEP}.  
\vspace{-3mm}

\subsection{Complexity Bound}\label{appen:complexity}
In this section, we derive the complexity bound of QNPE from Theorem~\ref{thm:main}.
Specifically, Let $N_{\epsilon}$ be the minimum number of iterations required by QNPE to achieve $\|\vx_N-\vx^*\|^2 \leq \epsilon$, and our goal is to upper bound $N_{\epsilon}$ in terms of the accuracy tolerance $\epsilon$. 
From the linear convergence result in Theorem~\ref{thm:main}, we have 
\begin{equation}\label{eq:linear_rate_restate}
  \|\vx_N-\vx^*\|^2 \leq \|\vx_0-\vx^*\|^2\left(1+\frac{\mu}{4L_1}\right)^{-N},
\end{equation}
and also from the superlinear convergence result we have 
\begin{equation}\label{eq:superlinear_restate}
  \|\vx_N-\vx^*\|^2 \leq \|\vx_0-\vx^*\|^2\left(1 + \frac{\mu}{4L_1}\sqrt{\frac{N}{N_\mathrm{tr}}}\right)^{-N},
\end{equation}
where $N_{\mathrm{tr}}$ is defined as
$
N_\mathrm{tr} \triangleq \frac{4}{3}+\frac{48}{L_1^2}\|\mB_0-\nabla^2 f(\vx^*)\|^2_F + \left(\frac{36}{L_1^2}+\frac{64}{3\mu L_1}\right)\!L_2^2\|\vx_0-\vx^*\|^2
$. 

From \eqref{eq:linear_rate_restate}, to ensure that $\|\vx_N-\vx^*\|^2 \leq \epsilon$, it is sufficient to have 
\begin{equation}\label{eq:compleixty_linear}
  \|\vx_0-\vx^*\|^2\left(1+\frac{\mu}{4L_1}\right)^{-N} \leq \epsilon \quad \Leftrightarrow \quad N \geq \frac{1}{\log(1+\frac{\mu}{4L_1})}\log\frac{\|\vx_0-\vx^*\|^2}{\epsilon}.
\end{equation}
On the other hand, from \eqref{eq:superlinear_restate}, 
it is sufficient to have 
\begin{equation}\label{eq:superlinear_inequality}
  \|\vx_0-\vx^*\|^2\left(1 + \frac{\mu}{4L_1}\sqrt{\frac{N}{N_\mathrm{tr}}}\right)^{-N} \leq \epsilon.
\end{equation}
To derive a bound on $N$ from \eqref{eq:superlinear_inequality}, we 
let $N^*$ be the number such that the inequality above becomes equality. Then \eqref{eq:superlinear_inequality} holds for all $N \geq N^*$.  
By using the elementary inequality $log(1+x) \leq x$ for $x\geq -1$, we have 
\begin{equation*}
  \log \frac{\|\vx_0-\vx^*\|^2}{\epsilon} = N^*\log \left(1 + \frac{\mu}{4L_1}\sqrt{\frac{N^*}{N_\mathrm{tr}}}\right) \leq \frac{\mu}{4L_1\sqrt{N_\mathrm{tr}}} (N^*)^{3/2},
\end{equation*} 
which implies that 
\begin{align*}
  N^* \geq \left( \frac{4 L_1 \sqrt{N_\mathrm{tr}}}{\mu} \log\frac{1}{\epsilon}\right)^{2/3}.
\end{align*}
Furthermore, we have  
\begin{equation*}
  \log \frac{\|\vx_0-\vx^*\|^2}{\epsilon} = N^*\log \left(1 + \frac{\mu}{4L_1}\sqrt{\frac{N^*}{N_\mathrm{tr}}}\right) \geq N^* \log \left(1+\left(\frac{\mu^2}{16L_1^2 N_\mathrm{tr}} \log\frac{1}{\epsilon}\right)^{1/3}\right),
\end{equation*}
which implies 
\begin{equation}\label{eq:compleixty_superlinear}
  N^* \leq \frac{1}{ \log \left(1+\left( \frac{\mu^2}{16L_1^2 N_\mathrm{tr}} \log\frac{1}{\epsilon}\right)^{1/3}\right)} \log\left(\frac{\|\vx_0-\vx^*\|^2}{\epsilon}\right).
\end{equation}
Thus, by combining \eqref{eq:compleixty_linear} and \eqref{eq:compleixty_superlinear}, we obtain 
\begin{equation*}
  N_{\epsilon} \leq \min\left\{ \frac{1}{\log(1+\frac{\mu}{4L_1})}, \frac{1}{ \log \left(1+\left( \frac{\mu^2}{16L_1^2 N_\mathrm{tr}} \log\frac{1}{\epsilon}\right)^{1/3}\right)} \right\}\log\frac{\|\vx_0-\vx^*\|^2}{\epsilon}.
\end{equation*}

\section{Experimental Details}\label{appen:experiments}
In this section, we provide more details on the dataset generation process and the implementation of gradient descent, BFGS and our proposed QNPE algorithm. 

\vspace{.5em}\noindent \textbf{Dataset generation.} We first randomly generate the underlying true feature vectors $\va^*_1,\dots,\va^*_n \in \mathbb{R}^{d-1}$ and the underlying true parameter $\vx^* \in \mathbb{R}^{d-1}$. Specifically, each entry of $\{\va^*_i\}_{i=1}^n$ and $\vx^*$ is drawn independently according to the standard normal distribution $\mathcal{N}(0,1)$. Then the $i$-th feature vector $\va_i$ and the corresponding label $y_i$ are given by
\begin{equation*}
    \va_i = 
    \begin{bmatrix}
    \va_i^* + \vn_i + \mathbf{1} \\
    1
    \end{bmatrix}
    \in \mathbb{R}^d \quad \text{and} \quad y_i = \mathrm{sign}(\langle \va_i^*, \vx^* \rangle) \in \{-1,+1\}, 
\end{equation*}
where $\vn_i \sim \mathcal{N}(0,\sigma^2 \mI)$ is the i.i.d.~Gaussian noise vector and $\mathbf{1}$ denotes the all-one vector. In the experiment we set $\sigma = 0.8$. 

\vspace{.5em}\noindent \textbf{Gradient descent.} The update rule is given by $\vx_{k+1} = \vx_k - \eta_k \nabla f(\vx_{k})$, where the step size $\eta_k$ is selected by a backtracking line search scheme. Specifically, we choose $\eta_k$ to be the largest step size in the set $\{\sigma_k \beta^i: i \geq 0\}$ that guarantees a sufficient decrease in the function value: 
\begin{equation*}
    f(\vx_k - \eta_k \nabla f(\vx_{k})) \leq f(\vx_k) - \frac{\eta_k}{2} \| \nabla f(\vx_k)\|^2.  
\end{equation*}
Moreover, we set $\eta_{k+1} = \eta_k/\beta$ for $k\geq 0$ similar to the strategy in Algorithm~\ref{alg:Full_EQN}. In the experiment, we set $\beta = 0.5$.  

\vspace{.5em}\noindent \textbf{BFGS.} We implemented the classical BFGS algorithm, where we employ the Mor\'e-Thuente line search scheme using the code by \cite{oleary1991matlab}. In the experiment, we set the initial Hessian approximation matrix as $\mB_0 = L_1 \mI$. 

\vspace{.5em}\noindent \textbf{Our proposed QNPE algorithm.} In the experiments, we set the line search parameters in Subroutine~\ref{alg:ls} by $\alpha_1 = \alpha_2 = \beta = 0.5$. We also set $\mB_0 = \mu \mI$ and $\rho = 1$ in Subroutine~\ref{alg:hessian_approx}. 
\end{document}

%% file: math_commands.tex
\newcommand\Vector[1]{\mathbf{#1}}

\newcommand\va{{\Vector{a}}}
\newcommand\vb{{\Vector{b}}}

\newcommand\vg{{\Vector{g}}}

\newcommand\vn{{\Vector{n}}}

\newcommand\vp{{\Vector{p}}}

\newcommand\vr{{\Vector{r}}}
\newcommand\vs{{\Vector{s}}}

\newcommand\vu{{\Vector{u}}}
\newcommand\vv{{\Vector{v}}}
\newcommand\vw{{\Vector{w}}}
\newcommand\vx{{\Vector{x}}}
\newcommand\vy{{\Vector{y}}}
\newcommand\vz{{\Vector{z}}}

\newcommand\MATRIX[1]{\mathbf{#1}}

\newcommand\mA{{\MATRIX{A}}}
\newcommand\mB{{\MATRIX{B}}}

\newcommand\mG{{\MATRIX{G}}}
\newcommand\mH{{\MATRIX{H}}}
\newcommand\mI{{\MATRIX{I}}}

\newcommand\mS{{\MATRIX{S}}}
\newcommand\mT{{\MATRIX{T}}}

\newcommand\mV{{\MATRIX{V}}}
\newcommand\mW{{\MATRIX{W}}}

\newcommand\mLambda{{\MATRIX{\Lambda}}}

\newcommand\sS{{\mathbb{S}}}

\newcommand\bigO{\mathcal{O}}
\newcommand\semiS{\mathbb{S}_{+}}

\newcommand\reals{\mathbb{R}}
\newcommand\op{{\mathrm{op}}}

\DeclareMathOperator{\sign}{sign}